\documentclass[a4paper,reqno,10pt]{amsart}

\usepackage{subfiles}
\usepackage{upgreek}
\usepackage{amsfonts}
\usepackage{amsfonts,graphicx,bezier,enumerate,multicol,multirow,subfloat,color,xcolor,amssymb,mathtools,verbatim,appendix,pdfpages,color,colortbl,lscape,longtable,euscript}
\usepackage{amsmath,amsthm,amssymb,colonequals}
\usepackage{alphalph,etoolbox}
\usepackage{indentfirst}
\usepackage{amscd}
\usepackage{setspace}
\usepackage{mathrsfs}
\usepackage{float}
\usepackage{graphicx, subcaption}
\usepackage{wrapfig}
\usepackage{enumitem}
\usepackage{colonequals}

\setlist[enumerate]{format=\normalfont}

\newtheorem{theorem}{Theorem}[section]
\newtheorem{prop}[theorem]{Proposition}
\newtheorem{lemma}[theorem]{Lemma}
\newtheorem{definition}[theorem]{Definition}
\newtheorem{cor}[theorem]{Corollary}

\theoremstyle{definition}

\newtheorem{example}[theorem]{Example}

\newtheorem{remark}[theorem]{Remark}

\newtheorem{notation}[theorem]{Notation}

\numberwithin{equation}{section}

\newcommand{\bmc}{\begin{multicols}}
	\newcommand{\emc}{\end{multicols}}
\usepackage[colorlinks]{hyperref}

\usepackage{tikz,mathrsfs}
\usetikzlibrary{arrows,decorations.pathmorphing,decorations.pathreplacing,positioning,shapes.geometric,shapes.misc,decorations.markings,decorations.fractals,calc,patterns,matrix}

\tikzset{
cvertex/.style={circle,draw=blue,inner sep=1pt,outer sep=3pt},
cvertex1/.style={circle,draw=green,inner sep=1pt,outer sep=3pt},
vertex/.style={circle,fill=black,inner sep=1pt,outer sep=3pt},
DBs/.style={circle,draw=black,circle,fill=black,inner sep=0pt, minimum size=3pt},
DB/.style={circle,draw=black,circle,fill=black,inner sep=0pt, minimum size=4pt},
DWs/.style={circle,draw=black,circle,fill=white,inner sep=0pt, minimum size=3pt},
DWds/.style={circle,draw=black,densely dotted,circle,fill=white,inner sep=0pt, minimum size=3pt},
DWds1/.style={circle,draw=white,fill=pink,inner sep=0pt, minimum size=3pt},
DWds2/.style={circle,draw=white,fill=purple,inner sep=0pt, minimum size=3pt},
DWds3/.style={circle,draw=white,fill=blue,inner sep=0pt, minimum size=3pt},
DWds4/.style={circle,draw=white,fill=red,inner sep=0pt, minimum size=3pt},
DWds5/.style={circle,draw=white,fill=green,inner sep=0pt, minimum size=3pt},
DW/.style={circle,draw=black,inner sep=0pt, minimum size=4pt},
tvertex/.style={inner sep=1pt,font=\scriptsize},
gap/.style={inner sep=0.5pt,fill=white},
Bgap/.style={circle,draw=black,circle,inner sep=-1pt,fill=black},
mid/.style={inner sep=0.5pt},
Ggap/.style={inner sep=0.5pt,fill=green!40!black!20}}

\newcommand{\Afour}[4]{%
\begin{tikzpicture}[baseline=(O.base),scale=0.175]
\node (O) at (0,-0.5) {};
\node at (0,0) [#1] {};
\node at (1,0) [#2] {};
\node[black!60!white] at (2.1,0)  {$..$};
\node at (3.2,0) [#3] {};
\node at (4.2,0) [#4] {};
\end{tikzpicture}	
}

\tikzset{
W/.style={circle,draw=black!60!white,circle,fill=white,inner sep=0pt, minimum size=4pt},
B/.style={circle,draw=black!60!white,circle,fill=black!60!white,inner sep=0pt, minimum size=4pt},
}

\renewcommand{\*}{%
\discretionary {\thinspace\the\textfont2\char1}{}{}%
}

\newcommand{\PB}{\mathrm{PBr}}

\newcommand{\scrA}{\EuScript{A}}
\newcommand{\scrB}{\EuScript{B}}
\newcommand{\scrC}{\EuScript{C}}

\newcommand{\scrH}{\EuScript{H}}

\newcommand{\scrJ}{\EuScript{J}}
\newcommand{\scrK}{\EuScript{K}}

\newcommand{\gen}{\mathbb{I}}
\DeclareMathOperator{\Spec}{Spec}

\begin{document}
	
\title{Pure Braid Group Presentations via Longest Elements}
\author{Caroline Namanya}
\address{Caroline Namanya, Makerere University, Kampala, Uganda \& University of Glasgow,  Glasgow, Scotland 
}
\email{caronamanya97@gmail.com}
\begin{abstract}
This paper gives  a new, simplified presentation of the classical pure braid group. The generators are given by  the squares of the longest elements over connected subgraphs, and we prove that the only relations are either commutators or certain palindromic length 5 box relations. This presentation is motivated by twist functors in algebraic geometry, but the proof is entirely Coxeter-theoretic. We also  prove  that the analogous  set does not  generate for all  Coxeter arrangements, which in particular answers a question of Donovan and Wemyss.
\end{abstract}
\maketitle
\parindent 20pt
\parskip 0pt

\section{Introduction}

The classical pure braid group, equivalently the pure braid group of type \(\mathrm{A}\), is a fundamental object in algebra, geometry and topology.  The purpose of this paper is to give a new, and simplified, presentation of this group using only squares of  longest elements over connected subgraphs, and  to then use this to answer questions motivated from algebraic geometry.   In the process we  place this presentation in the context of other fundamental groups.  Our methods are algebraic, and are  independent of   the geometric motivation. 

\subsection{The new presentation}
As recalled in Definition \ref{01 of 17 may 2022}, the classical pure braid group \( \mathrm{\PB}_{A_{n}}\) is the kernel of the natural surjection  \[  \mathrm{Br}_{n+1} \twoheadrightarrow  \mathfrak{S}_{n+1}, \]  and can also be viewed as \(\uppi_{1}\) of the type \(\mathrm{A}\) (complexified) hyperplane arrangement. Generators typically involve  a choice of looping around hyperplanes. Unfortunately, these choices often lead to non-symmetric and often unpleasant presentations. 

Both the need to give a nice presentation, and our geometric purposes, require a different generating set. Consider the  \(A_{n}\) Dynkin graph numbered
\[ \underset{\mathclap{1}}{ \bullet } - \underset{\mathclap{2}}{\bullet} - \dotsb - \underset{\mathclap{{n-1}}}{\bullet} - \underset{\mathclap{n}}{\bullet}\]   Then  for any  a connected  subgraph \( \scrA \subseteq  A_{n}\), consider \(  \ell_{ \scrA}^{2} \), where  \(\ell_{ \scrA}\) is the longest element over \( \scrA\). The following  is  our first result.
\begin{prop}[\ref{04 of 15 may 2022}]
The set \( \{ \ell_{ \scrA}^{2} \mid  \scrA \subseteq A_{n}, \,\scrA~ \text{connected}~ \}  \) generates \( \mathrm{\PB}_{A_{n}}.\)
\end{prop}
By  slight abuse of notation, write  $\scrA\colonequals \ell_{\scrA}^2$. Leading to our main result, consider connected subgraphs \(\scrA\) and \(\scrC\) of \(A_{n}\), then by the distance $d(\scrA,\scrC)$ we mean the number of edges between $\scrA$ and $\scrC$. The case \(d(\scrA,\scrC)=2\) corresponds to when  $\scrA$ and $\scrC$  are  precisely one node apart, namely
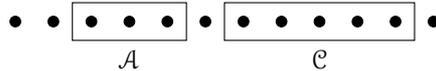
\[
\begin{tikzpicture}
\draw (0.75,0.25) -- (2.25,0.25)--(2.25,-0.25)--(0.75,-0.25)--cycle;
\draw (2.75,0.25) -- (5.25,0.25)--(5.25,-0.25)--(2.75,-0.25)--cycle;
\node at (1.5,-0.5) {$\scrA$};
\node at (4,-0.5) {$\scrC$};
\filldraw (0,0) circle (2pt);
\filldraw (0.5,0) circle (2pt);
\filldraw (1,0) circle (2pt);
\filldraw (1.5,0) circle (2pt);
\filldraw (2,0) circle (2pt);
\filldraw (2.5,0) circle (2pt);
\filldraw (3,0) circle (2pt);
\filldraw (3.5,0) circle (2pt);
\filldraw (4,0) circle (2pt);
\filldraw (4.5,0) circle (2pt);
\filldraw (5,0) circle (2pt);
\filldraw (5.5,0) circle (2pt);
\end{tikzpicture}		
\]
Given such a pair, we say that  a subgraph  $\scrB$ is compatible  with \( (\scrA, \scrC)\)  if   $\scrB$ is a connected subgraph of the following dotted area,  containing the red node.
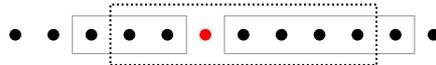
\[
\begin{tikzpicture}
\draw[white!60!black] (0.75,0.25) -- (2.25,0.25)--(2.25,-0.25)--(0.75,-0.25)--cycle;
\draw[white!60!black] (2.75,0.25) -- (5.25,0.25)--(5.25,-0.25)--(2.75,-0.25)--cycle;
\draw[densely dotted,line width=0.75pt] (1.25,0.4) -- (4.75,0.4)--(4.75,-0.4)--(1.25,-0.4)--cycle;
\filldraw (0,0) circle (2pt);
\filldraw (0.5,0) circle (2pt);
\filldraw (1,0) circle (2pt);
\filldraw (1.5,0) circle (2pt);
\filldraw (2,0) circle (2pt);
\filldraw[red] (2.5,0) circle (2pt);
\filldraw (3,0) circle (2pt);
\filldraw (3.5,0) circle (2pt);
\filldraw (4,0) circle (2pt);
\filldraw (4.5,0) circle (2pt);
\filldraw (5,0) circle (2pt);
\filldraw (5.5,0) circle (2pt);		
	
\end{tikzpicture}
\]
The following is our main result.
\begin{theorem} [\ref{Feb 2022}]\label{02 02 2022}
The pure braid group \( \mathrm{\PB}_{A_{n}}\)  has a presentation with generators given by connected subgraphs $\scrA \subseteq  A_n$, subject to the relations
\begin{enumerate}
\item  \( \scrA\cdot\scrB = \scrB\cdot \scrA\) if \(d(\scrA, \scrB)\geq 2\), or \(\scrA \subseteq \scrB\), or \(\scrB\subseteq\scrA\).
\item For all $\scrA$ and all $\scrC$ such that $d(\scrA,\scrC)=2$, then
\[
(\scrA\cup\scrB)^{-1}\cdot(\scrA\cdot\scrB\cdot\scrC)\cdot(\scrB\cup\scrC)^{-1}	=
(\scrC\cup\scrB)^{-1}\cdot(\scrC\cdot\scrB\cdot\scrA)\cdot(\scrB\cup\scrA)^{-1}
\]
for all  $\scrB$ with  compatible \((\scrA,\scrC)\).
\end{enumerate}		
\end{theorem}
We also give a presentation in terms of double-indices in Corollary \ref{07 of 15 may 2022}, which then makes it easier to compare to the other presentations in the literature. We remark that the above presentation is symmetric,  and furthermore there are precisely \(\binom{n+2}{5}\)  non-commutator relations, each of which is palindromic and has  degree 5. The above presentation is different to the presentations in \cite{BB,A,FV,DG}, and the geometric presentations in \cite{MM}. Using the exact sequence \cite[(4), p150]{GW}, Theorem~\ref{02 02 2022} also independently recovers the presentation of the pure mapping class group of the punctured 2-sphere discovered in the recent work of Hirose--Omori \cite[3.1]{HO}. 

To  prove that the set \(\{\ell_{\scrA}^{2}\} \) generates, it  suffices to show that the standard generators in \cite{A,FV} can be written as a product of the elements in the set \(\{\ell_{\scrA}^{2}\}. \)  The \(\{\ell_{\scrA}^{2}\}\) are symmetric  whereas the standard  relations  for \( \mathrm{\PB}_{A_{n}}\) (\cite{A,FV}, see e.g  \cite{BB}) come from the existence of a split  short exact sequence
\[1 \to \mathrm{F}_{n-1} \to \mathrm{\PB}_{A_{n}} \to \mathrm{\PB}_{A_{n-1}} \to 1, \] where \(\mathrm{F_{n-1}}\) is a  free group,  using an  inductive argument. To prove results on the relations in the new presentation, we track the standard relations in \cite{A,FV} under a homomorphism \(\upphi \) that is expressed in Proposition \ref{02 of 15 may 2022}, and show that  the standard relations  are mapped to the identity in the new presentation.  This part of the argument is much harder, since the splitting involves choice, whereas the new relations are symmetric.

\medskip 

We finally show in Corollary~\ref{cor:not gen} that the pure braid groups of other Coxeter arrangements are not in general generated by squares of longest elements, and so the above is largely a type \(A\) phenomena. We further explain  how this relates to monodromy  around high codimension walls in the corresponding hyperplane arrangement, and thus answer a question of Donovan-Wemyss \cite{DW3}.
\subsection*{Acknowledgements}This work forms part of the author's  PhD, and was partially funded by a GRAID scholarship, an IMU  Breakout Graduate Fellowship, and by the ERC Consolidator Grant 101001227 (MMiMMa). The author would like to thank her supervisors Michael Wemyss and David Ssevviiri for their helpful guidance, Genki Omori for explaining the connection to \cite{HO}, and the referee for helpful comments. 

\section{Preliminaries}
\subsection{Classical Presentation}\label{classical presentation}
The classical Artin braid group is defined to be 
\[
\mathrm{Br}_{n}\colonequals \left\langle s_{1}, \hdots ,s_{n-1} \left|
\begin{array}{ll}
s_{i}s_{j}=s_{j}s_{i}\mbox{ if }|i-j|\geq 2,\\
s_{i}s_{i+1}s_{i}=s_{i+1}s_{i}s_{i+1}\mbox{ for all }i=1,2, \hdots, n-2
\end{array}
\right.\right\rangle.
\]
\begin{definition}\label{01 of 17 may 2022}
The kernel of \(\mathrm{Br}_{n+1} \to \mathfrak{S}_{n+1} \) sending \(s_{i} \) to the permutation \((i,i+1)\)    is defined to be the \emph{pure braid group}, and will be written  \(\mathrm{\PB}_{A_{n}}.\)
\end{definition}
For any $i, j$ satisfying $1\leq  i<j \leq n+1$, set
\[
\upsigma_{i,j}=\upsigma_{j,i}=(s_{j-1}\hdots s_{i+1})s_{i}(s^{-1}_{i+1} \hdots s^{-1}_{j-1})\in\mathrm{Br}_{n+1}.
\]
According to \cite{A,FV},  a presentation of the pure braid group can be given as follows.  As generators, $ \mathrm{\PB}_{A_{n}}=\langle A_{i,j}=A_{j,i}=\upsigma^{2}_{i,j}\rangle$, subject to the relations
\[
\begin{split}
&A^{-1}_{r,s}A_{i,j}A_{r,s} \\
&=\left\{\begin{array}{ll}
A_{i,j} & \text{if}~ r < s < i <j ~\text{or}~ i < r < s <j,\\
A_{r,j}A_{i,j}A^{-1}_{r,j} & \text{if}~ r < s = i <j,\\
(A_{i,j}A_{s,j})(A_{i,j})(A_{i,j}A_{s,j})^{-1} & \text{if}~ r = i <s <j,\\
(A_{r,j}A_{s,j}A^{-1}_{r,j}A^{-1}_{s,j})A_{i,j}(A_{r,j}A_{s,j}A^{-1}_{r,j}A^{-1}_{s,j})^{-1}&\text{if}~ r < i <s <j.
\end{array} \right.
\end{split}
\]
\subsection{Generation  via longest elements squared}\label{section 2.2}
A connected subgraph of $A_n$ is determined by its leftmost vertex \(i\), and its rightmost vertex \(j\),  where \(i\leq j\). To such a subgraph is an associated longest element in the corresponding parabolic subgroup of the symmetric group generated by the subgraph. The standard lift of this element to  \(\mathrm{Br}_{A_{n}}=\mathrm{Br}_{n+1} \), will be  written \(\ell_{i,j}\) (see e.g  \cite[p4]{BT}, \cite[Lemma 9.1.10]{ECHLPT} and \cite[p2]{G}).

Reversing  words in the Artin generators, that is to say  reading words backwards gives an antiautomorphism \(\mathrm{Br}_{n+1}\to \mathrm{Br}_{n+1}\) which we will write as  \(g \mapsto \bar{g}\) (see e.g \cite{FDSM,G}).

\begin{lemma}\label{09 july 2021}
$\ell_{i,i}= s_{i}$  and further if \(i < j\) then  
{\scriptsize\begin{align*}
\ell_{i,j} &= (s_i)(s_{i+1}s_i)(s_{i+2}s_{i+1}s_i)\hdots(s_{j-1}\hdots s_i)(s_j\hdots s_i)= (s_i\hdots s_j)\ell_{i,j-1}=\ell_{i,j-1}(s_{j}\cdots s_{i})\\
&= (s_{j})(s_{j-1}s_{j})(s_{j-2}s_{j-1}s_{j})\hdots(s_{i+1}\hdots s_{j})(s_{i}\hdots s_{j}) = (s_{j}\hdots s_{i})\ell_{i+1,j}=\ell_{i+1,j}(s_{i}\hdots s_{j}).
\end{align*}}
\end{lemma}
\begin{proof}
The first equality is standard (see e.g \cite{D2}). The second equality holds by regrouping using  \(s_{i}s_{j}=s_{j}s_{i} \) whenever \(|i-j|\geq 2,\) to bring forward certain elements as follows
\[
\begin{tikzpicture}[>=stealth]
\node at (0,0) {$(s_{i}s_{i+1})\,\,\,(s_{i})(s_{i+2}s_{i+1}s_{i})(s_{i+3}s_{i+2}s_{i+1}s_{i})\hdots(s_j\hdots s_i)$};
\draw (-1.70,-0.25) -- (-1.70,-0.4) -- (-2.95,-0.4);
\draw[->] (-2.94,-0.4) -- (-2.94,-0.25);
\draw (0.15,-0.25) -- (0.15,-0.5) -- (-2.83,-0.5);
\draw[->] (-2.82,-0.5) -- (-2.82,-0.25);
\draw (2.9,-0.25) -- (2.9,-0.6) -- (-2.75,-0.6);
\draw[->] (-2.74,-0.6) -- (-2.74,-0.25);				
\end{tikzpicture}. 
\]
The  third  equality  follows, since \(\ell_{i,j-1} =(s_i)(s_{i+1}s_i)(s_{i+2}s_{i+1}s_i)\hdots(s_{j-1}\hdots s_i).\) Applying the antiautomorphism  \(g\to \bar{g},\) which fixes  \(\ell_{i,j-1},\) gives the third equality. The second line is simillar.
\end{proof}

\begin{cor}\label{2nd 07 july 2021} If \(i < j\) then  
\begin{align*}
\ell_{i,j}^{2}&= (s_{j} \ldots s_{i})(s_{i} \ldots s_{j})\ell_{i,j-1}^{2}
= (s_{i} \ldots s_{j})(s_{j} \ldots s_{i})\ell_{i+1,j}^{2}\\ 		
\ell_{i,j}^{2}&= \ell_{i,j-1}^{2}(s_{j} \ldots s_{i})(s_{i} \ldots s_{j})
= \ell_{i+1,j}^{2}(s_{i} \ldots s_{j})(s_{j} \ldots s_{i}).			 
\end{align*}
\end{cor}
\begin{proof}
This follows by repeated application of Lemma \ref{09 july 2021}. Indeed,
\begin{align*}
\ell_{i,j}^{2}&=(s_{j}\hdots s_{i})\ell_{i+1,j}(s_{i}\hdots s_{j})\ell_{i,j-1}\\
&=(s_{j}\hdots s_{i})\ell_{i,j}\ell_{i,j-1} \\
&=(s_{j}\hdots s_{i})(s_i\hdots s_j)\ell_{i,j-1}^{2} 
\end{align*}
and
\begin{align*}
\ell_{i,j}^{2}&=(s_i\hdots s_j)\ell_{i,j-1}(s_{j}\hdots s_{i})\ell_{i+1,j}\\
&=(s_i\hdots s_j)\ell_{i,j}\ell_{i+1,j}  \\
&=(s_i\hdots s_j)(s_{j}\hdots s_{i})\ell_{i+1,j}^{2}. 
\end{align*}
Similarly, \(\ell_{i,j}^{2}= \ell_{i,j-1}^{2}(s_{j} \ldots s_{i})(s_{i} \ldots s_{j})= \ell_{i+1,j}^{2}(s_{i} \ldots s_{j})(s_{j} \ldots s_{i}).\) 
\end{proof}

The above allows us to  exhibit a new  generating set for \(\mathrm{\PB}_{A_n}\). In what follows, to obtain a unified statement we adopt the convention that \(\ell_{i,j}= 1\)\ whenever \(j < i\).  As calibration in the statement below, this means that \(A_{i,i+1}=\ell_{i,i}^{2}.\)

\begin{prop}\label{02 of 15 may 2022} For all \(i < j, \)
\(A_{i,j}=\ell_{i,j-2}^{-2}\cdot \ell_{i,j-1}^{2} \cdot   \ell_{i+1,j-2}^{2} \cdot \ell_{i+1,j-1}^{-2}. \)
\end{prop}
\begin{proof}
Consider \( \ell_{i,j-2}^{2} \cdot A_{i,j} \cdot\ell_{i+1,j-1}^{2}\). By definition, this equals 
\begin{align*}
&=\ell_{i,j-2}^{2} \cdot \upsigma_{i,j}^{2} \cdot\ell_{i+1,j-1}^{2}\\
&=\ell_{i,j-2}^{2}\cdot (s_{j-1}\hdots s_{i+1})\cdot s_{i}^{2}\cdot(s^{-1}_{i+1} \hdots s^{-1}_{j-1})\cdot \ell_{i+1,j-1}^{2}\tag{by definition}\\
&=\ell_{i,j-2}^{2}\cdot (s_{j-1}\hdots s_{i+1})\cdot s_{i}^{2}\cdot (s^{-1}_{i+1} \hdots s^{-1}_{j-1})\cdot (s_{j-1}\hdots s_{i+1}) \cdot\ell_{i+2,j-1} \cdot\ell_{i+1,j-1}\tag{Lemma \ref{09 july 2021}}\\
&=\ell_{i,j-2}^{2}\cdot (s_{j-1}\hdots s_{i+1})\cdot s_{i}^{2} \cdot \ell_{i+2,j-1} \cdot \ell_{i+1,j-1}\\
&=\ell_{i,j-2} \cdot \ell_{i,j-2}\cdot (s_{j-1}\hdots  s_{i}) \cdot  s_{i} \cdot \ell_{i+2,j-1} \cdot (s_{i+1}\hdots  s_{j-1}) \ell_{i+1,j-2}\\
&=\ell_{i,j-2} \cdot \ell_{i,j-1} \cdot  s_{i} \cdot \ell_{i+1,j-1} \cdot  \ell_{i+1,j-2}\tag{Lemma \ref{09 july 2021}}\\
&=\ell_{i,j-2} \cdot (s_{j-1}\hdots  s_{i})  \cdot \ell_{i+1,j-1} \cdot  s_{i} \cdot (s_{i+1}\hdots  s_{j-1}) \cdot \ell_{i+1,j-2} \cdot  \ell_{i+1,j-2}\\
&=\ell_{i,j-1}  \cdot \ell_{i+1,j-1} \cdot  (s_{i}\hdots  s_{j-1}) \cdot \ell_{i+1,j-2} \cdot  \ell_{i+1,j-2}\\
&=\ell_{i,j-1}  \cdot \ell_{i,j-1} \cdot \ell_{i+1,j-2} \cdot  \ell_{i+1,j-2}\\
&=\ell_{i,j-1}^{2} \cdot \ell_{i+1,j-2}^{2}.	\qedhere					
\end{align*} 
\end{proof}	
\begin{cor}\label{04 of 15 may 2022}
 The set \( \{ \ell_{\scrA}^{2} \mid \scrA \subseteq A_{n}, \, \scrA~ \text{connected}~ \} \leq \mathrm{Br}_{A_{n}}\) generates the pure braid group \(\PB_{A_{n}}.\)
\end{cor}
\begin{proof}
Proposition  \ref{02 of 15 may 2022} shows that each element \(A_{i,j}\) of the generating  set of the pure braid group in \cite{A,FV}  can be written as the product of elements in the set \( \{ \ell_{\scrA}^{2} \mid \scrA \subseteq A_{n}, \, \scrA ~ \text{connected}\}. \)  Since each element \(\ell_{\scrA}^{2} \in \PB_{A_{n}} \), the result follows.
\end{proof}

\section{The new relations}
In this section, we first show in \textsection  \ref{sectiona}  that certain commutator and box relations hold, then in \textsection \ref{sectionb} we  prove that these suffice to give a full presentation of \(\mathrm{\PB}_{A_{n}}\).
\subsection{Box and commutator relations}\label{sectiona}
As notation, set \(A_{n}:= \Afour{B}{B}{B}{B}\). By \( \Afour{B}{B}{B}{W}\) we mean the connected subgraph starting at \(1\) and ending at \(n-1\).
\begin{lemma}\label{11 february 2022}
If \(\scrK \subseteq A_{n}  \) is  connected, then the following statements hold
\begin{enumerate}
\item [(1)] \((s_{n}\hdots s_{1})(s_{1}\hdots s_{n})\) commutes with \(\ell_\scrK\) for all   \(\scrK \subseteq \Afour{B}{B}{B}{W}\) 
\item[(2)] \((s_{1}\hdots s_{n})(s_{n}\hdots s_{1})\) commutes with \(\ell_\scrK\) for all  \(\scrK \subseteq \Afour{W}{B}{B}{B}\)			 
\end{enumerate}	
\end{lemma}
\begin{proof}	
Say  \(\scrK\) starts  at vertex \(i\) and ends at vertex \(j\), with   \(i\leq j < n\). Then
\begin{align*}
\ell_\scrK(s_{n}\hdots s_{1})(s_{1}\hdots s_{n})
&= \ell_{i,j}(s_{n}\hdots s_{1})(s_{1}\hdots s_{n})\\
&= s_{n}\hdots s_{j+2}(\ell_{i,j}\,s_{j+1}\hdots s_{i})s_{i-1} \hdots s_{1}s_{1}\hdots s_{n}\tag{commutativity of braids }\\
&= s_{n}\hdots s_{j+2}(s_{j+1}\hdots s_{i} \,\ell_{i+1,j+1})s_{i-1} \hdots s_{1}s_{1}\hdots s_{n} \tag{Lemma \ref{09 july 2021}, both brackets equal to \(\ell_{i,j+1}\)}\\
&= s_{n}\hdots s_{i} \,\ell_{i+1,j+1}(s_{i-1} \hdots s_{1}s_{1} \hdots s_{i-1}) s_{i} \hdots s_{n} \\
&= s_{n}\hdots s_{i}(s_{i-1} \hdots s_{1}s_{1} \hdots s_{i-1})\ell_{i+1,j+1}\, s_{i} \hdots s_{n} \tag{commutativity of braids }\\
&= s_{n} \hdots s_{1}s_{1} \hdots s_{i-1}(\ell_{i+1,j+1}\, s_{i} \hdots s_{j+1})s_{j+2} \hdots s_{n} \\
&= s_{n} \hdots s_{1}s_{1} \hdots s_{i-1}( s_{i} \hdots s_{j+1}\,\ell_{i,j})s_{j+2} \hdots s_{n} \tag{Lemma \ref{09 july 2021}, both brackets equal to \(\ell_{i,j+1}\)}\\
&=(s_{n}\hdots s_{1})(s_{1}\hdots s_{n})\ell_{i,j}\tag{commutativity of braids }\\
&=(s_{n}\hdots s_{1})(s_{1}\hdots s_{n})\ell_\scrK.			
\end{align*}			
The other statement is similar.	
\end{proof}
\begin{lemma}\label{03 march 2022}
If  \(i < j\) then  \( \ell_{i,j}^{2}=(s_{j} \hdots s_{i})^{(j-i) +2}=(s_{i} \hdots s_{j})^{(j-i)+2}.\)
\end{lemma}
\begin{proof}
We will  prove  the case when \(i=1\) and \(j=n\) since the notation for this is clearer.
Recall  by Lemma \ref{09 july 2021} that \(\ell_{1,n} = (s_{n}\hdots s_{1})(s_{n} \hdots s_{2})(s_{n}\hdots s_{3})\hdots (s_{n}s_{n-1}) s_{n},
\) and also that \(\ell_{1,n} =(s_1)(s_2s_1)(s_3s_2s_1)\hdots(s_{n-1}\hdots s_1)(s_n\hdots s_1)
.\) Given this, 
\begin{align*}
&\ell_{1,n}^{2}\\
&=[(s_{n}\hdots s_{1})(s_{n} \hdots s_{2})(s_{n}\hdots s_{3})\hdots (s_{n}s_{n-1}) s_{n} \cdot (s_1)(s_2s_1)\hdots(s_{n-1}\hdots s_1)(s_n\hdots s_1)\\
&=[(s_{n}\hdots s_{1})(s_{n} \hdots s_{1})(s_{n}\hdots s_{3})\hdots (s_{n}s_{n-1})  \cdot(s_2s_1)\hdots(s_{n}\hdots s_1)(s_n\hdots s_1) \tag{\(s_{1}~\text{and} ~s_{n} \) commute through}\\
&=(s_{n}\hdots s_{1})(s_{n} \hdots s_{2}s_{1})\hdots (s_{n} \hdots s_{2}s_{1})(s_{n} \hdots s_{2}s_{1})\tag{repeat the above step}\\
&=(s_{n} \hdots s_{2}s_{1})^{n+1}.
\end{align*}
The general case is similar.
\end{proof}
The following technical Lemma will be required later.

\begin{lemma}\label{2nd of 03 march 2022}
For all \(i\geq 1\), and $j\geq i+2$, the following statements hold.
\begin{enumerate}
\item \( \begin{aligned}[t]
 \scriptstyle{ (s_{j} \hdots s^{2}_{i}\hdots s_{j})(s_{j-1} \hdots s^{2}_{i}\hdots s_{j-1})=}&  \scriptstyle{ (s_{j}\hdots s_{i+1})^{2}s_{i}s_{i+1}(s_{i+1}s_{i })(s_{i+2}s_{i+1})\hdots  (s_{j-1}s_{j-2})(s_{j}s_{j-1}).}\\
\end{aligned}\)
\item\(\begin{aligned}[t]
\scriptstyle{ (s_{i} \hdots s^{2}_{j}\hdots s_{i})(s_{i+1} \hdots s^{2}_{j}\hdots s_{i+1})=}&  \scriptstyle{(s_{i}\hdots s_{j-1})^{2}s_{j}s_{j-1}(s_{j-1}s_{j })(s_{j-2}s_{j-1})\hdots (s_{i+1}s_{i+2})(s_{i}s_{i+1}).}
\end{aligned}\)
\item\(\begin{aligned}
\scriptstyle{(s_{j-1} \hdots s^{2}_{i}\hdots s_{j-1})(s_{j} \hdots s^{2}_{i}\hdots s_{j})=} & \scriptstyle{ (s_{j-1}s_{j})(s_{j-2}s_{j-1 })\hdots (s_{i+1}s_{i+2})(s_{i}s_{i+1})s_{i+1}s_{i}(s_{i+1}\hdots s_{j})^{2}.}
\end{aligned}\)
\item\(\begin{aligned}
 \scriptstyle { (s_{i+1} \hdots s^{2}_{j}\hdots s_{i+1})(s_{i} \hdots s^{2}_{j}\hdots s_{i})=} & \scriptstyle{(s_{i+1}s_{i})(s_{i+2}s_{i+1 }) \hdots (s_{j-1}s_{j-2})(s_{j}s_{j-1}) s_{j-1}s_{j}(s_{j-1}\hdots s_{i})^{2}.}	
\end{aligned}\)
\end{enumerate}
\end{lemma}
\begin{proof}
For \((1)\), observe that
\begin{align*}
(s_{j} \hdots s^{2}_{i}\hdots s_{j})(s_{j-1} \hdots s^{2}_{i}\hdots s_{j-1})
&=(s_{j} \hdots s_{i+1})s^{2}_{i}\hdots s_{j-1}s_{j}s_{j-1} \hdots s^{2}_{i}\hdots s_{j-1}\\
&=(s_{j} \hdots s_{i+1})s^{2}_{i}\hdots s_{j}s_{j-1}s_{j} \hdots s^{2}_{i}\hdots s_{j-1}\\
&=(s_{j} \hdots s_{i+1})s_{j}s^{2}_{i}\hdots s_{j-2}s_{j-1}s_{j-2} \hdots s^{2}_{i}\hdots (s_{j}s_{j-1}).
\end{align*}
This process is repeated until we achieve the desired expression. The other statements are similar.
\end{proof}
As in the introduction, consider the graph \(A_{n}\), with connected subgraphs \(\scrA, \scrB\), \(\scrC.\)
\begin{definition}
If $\scrA\cap \scrB\neq\emptyset$ then define $d(\scrA,\scrB)=0$. Else, the distance $d(\scrA,\scrB)$ is defined to be the number of edges between $\scrA$ and $\scrB$.
\end{definition}
\begin{notation}\label{03 april 2022}
Fix \(\scrA, \scrC\) with \(d(\scrA, \scrC)=2\) and \(\scrB\) compatible with \((\scrA, \scrC)\) in the sense of the introduction. Equivalently, writing \(\scrA=[i,j], \scrC= [j+2,p]\) and \(\scrB=[a,k]\), we have
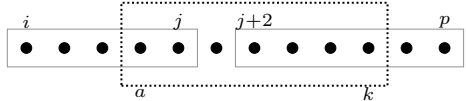
\[
\begin{tikzpicture}
\draw[white!60!black] (-0.25,0.25) -- (2.25,0.25)--(2.25,-0.25)--(-0.25,-0.25)--cycle;
\draw[white!60!black] (2.75,0.25) -- (5.75,0.25)--(5.75,-0.25)--(2.75,-0.25)--cycle;
\draw[densely dotted,line width=0.75pt] (1.25,0.6) -- (4.75,0.6)--(4.75,-0.5)--(1.25,-0.5)--cycle;
\filldraw (0,0) circle (2pt);
\filldraw (0.5,0) circle (2pt);
\filldraw (1,0) circle (2pt);
\filldraw (1.5,0) circle (2pt);
\filldraw (2,0) circle (2pt);
\filldraw (2.5,0) circle (2pt);
\filldraw (3,0) circle (2pt);
\filldraw (3.5,0) circle (2pt);
\filldraw (4,0) circle (2pt);
\filldraw (4.5,0) circle (2pt);
\filldraw (5,0) circle (2pt);
\filldraw (5.5,0) circle (2pt);	
\node at (0.0,0.35) {$\scriptstyle i$};
\node at (2,0.35) {$\scriptstyle j$};
\node at (3,0.35) {$\scriptstyle j+2$};	
\node at (5.5,0.35) {$\scriptstyle p$};
\node at (1.5,-0.60) {$\scriptstyle a$};
\node at (4.5,-0.60) {$\scriptstyle k$};			
\end{tikzpicture}
\]
The choice of such \(\scrA, \scrB\)  and \( \scrC \) translates into the condition \(1\leq i<a\leq j+1 \leq k < p \leq n.\) Given   \(\scrA, \scrB\) and \( \scrC \), set \(x\colonequals k-j\) and \(y\colonequals j+2 -a\), which visually are 
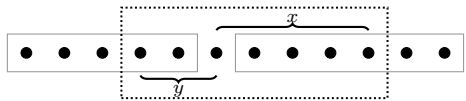
\[
\begin{tikzpicture}
\draw[white!60!black] (-0.25,0.25) -- (2.25,0.25)--(2.25,-0.25)--(-0.25,-0.25)--cycle;
\draw[white!60!black] (2.75,0.25) -- (5.75,0.25)--(5.75,-0.25)--(2.75,-0.25)--cycle;
\draw[densely dotted,line width=0.75pt] (1.25,0.60) -- (4.75,0.60)--(4.75,-0.60)--(1.25,-0.60)--cycle;
\filldraw (0,0) circle (2pt);
\filldraw (0.5,0) circle (2pt);
\filldraw (1,0) circle (2pt);
\filldraw (1.5,0) circle (2pt);
\filldraw (2,0) circle (2pt);
\filldraw (2.5,0) circle (2pt);
\filldraw (3,0) circle (2pt);
\filldraw (3.5,0) circle (2pt);
\filldraw (4,0) circle (2pt);
\filldraw (4.5,0) circle (2pt);
\filldraw (5,0) circle (2pt);
\filldraw (5.5,0) circle (2pt);		
\draw [
thick,
decoration={
	brace,
	raise=0
},
decorate
] (2.5,0.30) -- (4.5,0.30) 
node [pos=0.5,anchor=north,yshift=10] {$\scriptstyle x$};		
\draw [
thick,
decoration={
	brace,
	mirror,
	raise=0
},
decorate
] (1.5,-0.30) -- (2.5,-0.30) 
node [pos=0.5,anchor=north,yshift=0] {$\scriptstyle y$};	
	
\end{tikzpicture}
\]
Then for \(b \colonequals i+ (k-(j+1))\) and \(h \colonequals a+(p-(j+1))\) define 
\begin{align*}
\mathfrak{c}_{a-1,b}&\colonequals
(s_{a-1} \hdots s_{a+x-2})(s_{a-2} \hdots s_{a+x-3})\hdots (s_{i+1} \hdots s_{b+1})(s_{i} \hdots s_{b})\\
\mathfrak{c}_{j+1,x+a-1}&\colonequals
(s_{j+1} \hdots s_{k})(s_{j} \hdots s_{k-1})\hdots (s_{a+1} \hdots s_{x+a})(s_{a} \hdots s_{x+a-1})\\ 
\mathfrak{d}_{k+1,h}&\colonequals
(s_{k+1} \hdots s_{k-y+2})(s_{k+2} \hdots s_{k-y+3})\hdots (s_{p-1} \hdots s_{h-1})(s_{p} \hdots s_{h})\\
\mathfrak{d}_{j+1,k-y+1}&\colonequals
(s_{j+1} \hdots s_{a})(s_{j+2} \hdots s_{a+1})\hdots (s_{k-1} \hdots s_{k-y})(s_{k} \hdots s_{k-y+1})			
\end{align*}
where  for \(\mathfrak{c}\)  each factor has length \(x\), and for \(\mathfrak{d}\)  each factor has length \(y\).
\end{notation}
Note that \(\mathfrak{c}_{\upalpha, \upbeta}\)  should be understood as starting at \(s_{\upalpha}\) and ending at \(s_{\upbeta}\), where each factor has length \(x\), and  the indices decrease by one at each step. The notation \(\mathfrak{d}_{\upalpha,\upbeta}\) should be understood similarly, but each factor has length \(y\) and each step increases the indices.

The following is a general  version of Lemma \ref{2nd of 03 march 2022},  which covered the cases \(x=2\) and \(y=2\) respectively. Recall from $\S \ref{section 2.2}$ that  \(\bar{\mathfrak{c}}_{a-1,b}\) will be the expression for \(\mathfrak{c}_{a-1,b}\) read backwards, namely  \(\bar{\mathfrak{c}}_{a-1,b}=(s_{b} \hdots s_{i})(s_{b+1} \hdots s_{i+1})\hdots (s_{a+x-3} \hdots s_{a-2})(s_{a+x-2} \hdots s_{a-1}) \) with the obvious variations for \(\bar{\mathfrak{d}}.\) 
\begin{lemma}\label{01 of 22 june 2022}
For all \(1\leq i<a\leq j+1 \leq k < p \leq n\), as in Notation \textnormal{\ref{03 april 2022}} set \(x\colonequals k-j\) and \(y \colonequals  j+2 -a \). Then for \(b \colonequals i+ (k-(j+1))\) and \(h \colonequals a+(p-(j+1))\),  the following hold.
\begin{enumerate}
\item \( \begin{aligned}[t]
\underbrace{(s_{k}\hdots s^{2}_{i}\hdots s_{k})\hdots (s_{j+1}\hdots s^{2}_{i}\hdots s_{j+1})}_{x\,\mbox{\textnormal{\scriptsize terms}}} & = \big(s_{k}\hdots s_{a}\big)^{x}\cdot \mathfrak{c}_{a-1,b} \cdot \bar{\mathfrak{c}}_{a-1,b} \cdot \bar{\mathfrak{c}}_{j+1, x+a-1}
\end{aligned}\)
\item \( \begin{aligned}[t]
\underbrace{(s_{j+1}\hdots s^{2}_{p}\hdots s_{j+1})\hdots (s_{a}\hdots s^{2}_{p}\hdots s_{a})}_{y\,\mbox{\textnormal{\scriptsize terms}}} & = \mathfrak{d}_{j+1,k-y+1} \cdot  \mathfrak{d}_{k+1, h}\cdot  \bar{\mathfrak{d}}_{ k+1,h}\cdot  \big(s_{k}\hdots s_{a}\big)^{y}
\end{aligned}\)
\item \( \begin{aligned}[t]
\underbrace{(s_{a}\hdots s^{2}_{p}\hdots s_{a})\hdots (s_{j+1}\hdots s^{2}_{p}\hdots s_{j+1})}_{y\,\mbox{\textnormal{\scriptsize terms}}}  &=\big(s_{a}\hdots s_{k}\big)^{y} \cdot \mathfrak{d}_{k+1,h} \cdot  \bar{\mathfrak{d}}_{k+1, h}\cdot  \bar{\mathfrak{d}}_{ j+1,k-y+1} 
\end{aligned}\)
\item \( \begin{aligned}[t]
\underbrace{(s_{j+1}\hdots s^{2}_{i}\hdots s_{j+1})\hdots (s_{k}\hdots s^{2}_{i}\hdots s_{k})}_{x\,\mbox{\textnormal{\scriptsize terms}}} & = \mathfrak{c}_{j+1,x+a-1} \cdot  \mathfrak{c}_{a-1, b}\cdot  \bar{\mathfrak{c}}_{a-1, b}\cdot \big(s_{a}\hdots s_{k}\big)^{x}.
\end{aligned}\)
\end{enumerate} 
\end{lemma}
\begin{proof}
(1)~Consider  \((s_{k}\hdots s^{2}_{i}\hdots s_{k})\hdots (s_{j+1}\hdots s^{2}_{i}\hdots s_{j+1})\), where there are the \(x\) products. If  \(x=1\) then \(k=j+1\) and \(b=i\), so \begin{align*}\underbrace{(s_{k}\hdots s^{2}_{i}\hdots s_{k})\hdots (s_{j+1}\hdots s^{2}_{i}\hdots s_{j+1})}_{x\,\mbox{\scriptsize terms}}
&=s_{j+1}\hdots s^{2}_{b}\hdots s_{j+1}\\
&=(s_{j+1}\hdots s_{a})(s_{a-1} \hdots  s_{b})(s_{b}\hdots s_{a-1}) (s_{a} \hdots s_{j+1})\\
&=\big(s_{j+1}\hdots s_{a}\big)^{1}\cdot\mathfrak{c}_{a-1,b} \cdot \bar{\mathfrak{c}}_{a-1,b} \cdot \bar{\mathfrak{c}}_{j+1, x+a-1}.
\end{align*}
When \( x=2\) then \(k=j+2\)  and  the result follows by  Lemma \ref{2nd of 03 march 2022} since the expression  \((s_{k}\hdots s^{2}_{i}\hdots s_{k})\hdots (s_{j+1}\hdots s^{2}_{i}\hdots s_{j+1})\) equals  \[\scriptstyle{\big(s_{j+2}\hdots s_{a}\big)^{2}\big((s_{a-1}s_{a}) (s_{a-2}s_{a-1}) \hdots (s_{i}s_{i+1})\big)\big((s_{i+1}s_{i})\hdots (s_{a-1}s_{a-2})(s_{a}s_{a-1})\big)\big((s_{a+1}s_{a}) \hdots  (s_{j+2}s_{j+1})\big), }\] since \(b=i+1\). Now for general \(x\), repeating the proof of Lemma \ref{2nd of 03 march 2022}, in a similar way the product of \(x\) terms \((s_{k}\hdots s^{2}_{i}\hdots s_{k})\hdots (s_{j+1}\hdots s^{2}_{i}\hdots s_{j+1})\) equals
\[\scriptstyle{\big(s_{k}\hdots s_{a}\big)^{x}\big((s_{a-1}\hdots s_{a+x-2}) \hdots ( s_{i}\hdots s_{b})\big)\big((s_{b}\hdots s_{i})\hdots (s_{a+x-2} \hdots s_{a-1})\big)\big((s_{a+x-1} \hdots s_{a}) \hdots  (s_{k} \hdots s_{j+1})\big)}.\]
(2)~ Consider  \((s_{j+1}\hdots s^{2}_{p}\hdots s_{j+1})\hdots (s_{a}\hdots s^{2}_{p}\hdots s_{a})\), where there are  \(y\) products. If \(y=1\), then 
\(a=j+1\) and \(h=p\), so
\begin{align*}
\underbrace{(s_{j+1}\hdots s^{2}_{p}\hdots s_{j+1})\hdots (s_{a}\hdots s^{2}_{p}\hdots s_{a})}_{y\,\mbox{\scriptsize terms}}&=s_{j+1}\hdots s^{2}_{p}\hdots s_{j+1}\\
&=(s_{j+1}\hdots s_{k})(s_{k+1}\hdots s_{p})(s_{p}\hdots s_{k+1})(s_{k}\hdots s_{j+1})\\
&=\mathfrak{d}_{j+1,k-y+1} \cdot  \mathfrak{d}_{k+1, h}\cdot  \bar{\mathfrak{d}}_{ k+1,h}\cdot \big(s_{k}\hdots s_{j+1}\big)^{1}.
\end{align*}
When \( y=2\) then \(a=j\)  and  the result follows since  by  Lemma \ref{2nd of 03 march 2022} the statement  \((s_{j+1}\hdots s^{2}_{p}\hdots s_{j+1})\hdots (s_{a}\hdots s^{2}_{p}\hdots s_{a})\) equals
 \[\scriptstyle{\big((s_{j+1}s_{j})(s_{j+2}s_{j+1})\hdots (s_{k}s_{k-1})\big) \big((s_{k+1}s_{k}) \hdots (s_{p}s_{p-1})\big)\big((s_{p-1}s_{p})\hdots (s_{k+1}s_{k+2})(s_{k}s_{k+1})\big)\big(s_{k}\hdots s_{j}\big)^{2}, }\] since \(h=p-1\).
Now for general  \(y\), repeating the proof of Lemma \ref{2nd of 03 march 2022}, in a similar way the product of \(y\) terms \((s_{j+1}\hdots s^{2}_{p}\hdots s_{j+1})\hdots (s_{a}\hdots s^{2}_{i}\hdots s_{a})\) equals 
\[\scriptstyle{\big((s_{j+1}\hdots s_{a}) \hdots (s_{k}\hdots s_{k-y+1})\big)\big(s_{k+1}\hdots s_{k-y+2}) \hdots (s_{p}\hdots s_{h})\big)\big((s_{h}\hdots s_{p})\hdots (s_{k-y+2} \hdots s_{k+1})\big)\big(s_{k }\hdots s_{a}\big)^{y}}.\]
 The statements (3)  and (4) follow by applying  the antiautomorphism \(g\mapsto \bar{g},\) that is to say by reading (1) and (2) backwards.
\end{proof}
Leading into the next results, observe that \(a+x-1=k-y+1,\) as both equal \(k-j +a-1.\) 
\begin{cor} \label{02 of 22 june 2022}
With notation as in  Notation \textnormal{\ref{03 april 2022}} 
\(\mathfrak{c}_{j+1, a+x-1}=\mathfrak{d}_{j+1,k-y+1}\)	 and furthermore  \(\bar{\mathfrak{c}}_{j+1,a+x-1}= \bar{\mathfrak{d}}_{j+1,k-y+1}.\)
\end{cor}
\begin{proof}
When \(x=y=1\), or equivalently when  \( a=k=j+1\) , then    \(\mathfrak{c}_{j+1,a+x-1}=\mathfrak{c}_{j+1,j+1} =s_{j+1}=\mathfrak{d}_{j+1,j+1}= \mathfrak{d}_{j+1,k-y+1}.\)

When \(x,y \geq 2, \) by pulling the first element in each factor to the left as follows
\[
\begin{tikzpicture}[>=stealth]
\node at (0,0) {$\mathfrak{c}_{j+1,a+x-1}=(s_{j+1}\,\,\,\,\,  s_{j+2} \hdots s_{k})(s_{j}s_{j+1}\hdots s_{k-1})\hdots (s_{a}\hdots s_{a+x-1}),$};
\draw (-0.06,-0.15) -- (-0.06,-0.4) -- (-2.08,-0.4);
\draw[->] (-2.08,-0.4) -- (-2.08,-0.15);
\draw (2.3,-0.15) -- (2.3,-0.5) -- (-2.0,-0.5);
\draw[->] (-2.0,-0.5) -- (-2.0,-0.15);
\draw (3,-0.15) -- (3,-0.6) -- (-1.92,-0.6);
\draw[->] (-1.92,-0.6) -- (-1.92,-0.15);					
\end{tikzpicture} 
\]
we see that \(\mathfrak{c}_{j+1,a+x-1}=(s_{j+1}s_{j}  \hdots s_{a})(s_{j+2}\hdots s_{k})(s_{j+1}\hdots s_{k-1})\hdots (s_{a-1}\hdots s_{a+x-1}).\) Repeating, pulling again the following to the left 
\[
\begin{tikzpicture}[>=stealth]
\node at (0,0) {$\mathfrak{c}_{j+1,a+x-1}=(s_{j+1}s_{j}  \hdots s_{a})(s_{j+2}\,\,\,\,\hdots s_{k})(s_{j+1}\hdots s_{k-1})\hdots (s_{a-1}\hdots s_{a+x-1}),$};
\draw (0.65,-0.15) -- (0.65,-0.4) -- (-0.7,-0.4);
\draw[->] (-0.7,-0.4) -- (-0.7,-0.15);
\draw (3.4,-0.15) -- (3.4,-0.5) -- (-0.6,-0.5);
\draw[->] (-0.6,-0.5) -- (-0.6,-0.15);					
\end{tikzpicture} 
\] 
we see that \(\scriptstyle{\mathfrak{c}_{j+1,a+x-1}=(s_{j+1}s_{j}  \hdots s_{a})(s_{j+2}s_{j+1}\hdots s_{a-1})(s_{j+3}\hdots s_{k})(s_{j+2}\hdots s_{k-1})\hdots (s_{a-2}\hdots s_{a+x-1})}.\) Repeating,  \(\mathfrak{c}_{j+1,a+x-1}\) can be written as a product, each factor of length \(y\).  By definition, this is \(\mathfrak{d}_{j+1,a+x-1}=\mathfrak{d}_{j+1,k-y+1}.\)  The final statement follows by applying  the antiautomorphism \(g\mapsto \bar{g}.\)
\end{proof}
The following result proves various relations hold between the squares of longest elements over  connected subgraphs. The relations are  either  `far away commutativity' when there exists at least two edges between the subgraphs, `inclusion commutativity' when one of the subgraphs is contained in the other, or length five palindromic relations which we refer to as the box relations (as explained in Remark \ref{01 of 08 july 2022}).
\begin{prop}\label{3 of 29 january 2022}
Let  \(\scrJ,\scrK \subseteq A_{n}\) be connected. Then the following hold
\begin{enumerate}
\item \(\ell_{\scrJ}^{2}\ell_{\scrK}^{2}=\ell_{\scrK}^{2}\ell_{\scrJ}^{2}\) if \(d(\scrJ,\scrK)\geq 2\).
\item \(\ell_{\scrJ}^{2}\ell_{\scrK}^{2}=\ell_{\scrK}^{2}\ell_{\scrJ}^{2}\) if \(\scrJ \subseteq \scrK\) or \(\scrK \subseteq \scrJ.\)
\item There is an equality \[\ell_{i,k}^{-2}\,\ell_{i,j}^{2}\,\ell_{a,k}^{2}\,\ell_{j+2,p}^{2}\,\ell_{a,p}^{-2}=\ell_{a,p}^{-2}\,\ell_{j+2,p}^{2}\,\ell_{a,k}^{2}\,\ell_{i,j}^{2}\,\ell_{i,k}^{-2}\] whenever \(1\leq i<a\leq j+1 \leq k < p \leq n\).
	\end{enumerate}		
\end{prop}			

\begin{proof}
(1) Since \(\ell_{\scrJ}^{2}\) consists of only \(s_{j}\) with  \(j \in \scrJ\),  and \(\ell_{\scrK}^{2}\) consists of only \(s_{k}\) with  \(k \in \scrK\), the result follows from braid relations, since by  assumption  \(s_{j}\) commutes with \(s_{k}\) whenever \(j\in \scrJ\)  and \(k\in \scrK\).\\
(2) Without loss of generality we can consider the case  \(\scrJ \subsetneq \scrK.\)  The statement follows since \(\ell_{\scrK}^{2}\) is central in the parabolic subgroup of \(\mathrm{Br}_{A_{n}}\) generated by \(\scrK\) (see  e.g \cite[Theorem 7]{G}).
(3) Recall from Lemma~\ref{03 march 2022} that
\begin{align}
\ell_{a,k}^{2}&= (s_{k} \hdots s_{a})^{(k-a)+2}= (s_{a} \hdots s_{k})^{(k-a)+2}\label{longest symm}.
\end{align}
Using Corollary \ref{2nd 07 july 2021} repeatedly, we can factor \(\ell_{i,k}^{-2}\) and \(\ell_{a,p}^{-2}\) to obtain
\begin{align}
\ell_{i,k}^{-2}&= \big(\ell_{i,j}^{2}\underbrace{(s_{k}\hdots s^{2}_{i}\hdots s_{k})\hdots (s_{j+1}\hdots s^{2}_{i}\hdots s_{j+1})}_{k-j\,\mbox{\scriptsize terms}} \big)^{-1}  \label{eqn 1}\\
\ell_{a,p}^{-2} &=\big(\underbrace{(s_{j+1}\hdots s^{2}_{p}\hdots s_{j+1})\hdots (s_{a}\hdots s^{2}_{p}\hdots s_{a}}_{j+2-a\,\mbox{\scriptsize terms}} )\ell_{j+2,p}^{2}\big)^{-1} \label{eqn 2}
\end{align}	
Substituting in \eqref{eqn 1} and \eqref{eqn 2}, then $\ell_{i,k}^{-2}\ell_{i,j}^{2}\ell_{a,k}^{2}\ell_{j+2,p}^{2}\ell_{a,p}^{-2}$ equals
\[
\big((s_{k}\hdots s^{2}_{i}\hdots s_{k})\hdots (s_{j+1}\hdots s^{2}_{i}\hdots s_{j+1}) \big)^{-1}\ell_{a,k}^{2}\big((s_{j+1}\hdots s^{2}_{p}\hdots s_{j+1})\hdots (s_{a}\hdots s^{2}_{p}\hdots s_{a})\big)^{-1}.
		\]
By applying   Lemma \ref{01 of 22 june 2022} to the outer terms, and \eqref{longest symm} to the middle term, the above displayed equation equals 
\[\scriptstyle{\big((s_{k}\hdots s_{a})^{x} \cdot \mathfrak{c}_{a-1,b} \cdot \bar{\mathfrak{c}}_{a-1,b}\cdot \bar{\mathfrak{c}}_{j+1,a+x-1}\big)^{-1} \cdot(s_{k}\hdots s_{a})^{x+y} \cdot \big(\mathfrak{d}_{j+1,k-y+1} \cdot  \mathfrak{d}_{k+1,h} \cdot \bar{\mathfrak{d}}_{k+1,h}\cdot (s_{k}\hdots s_{a})^{y}\big)^{-1}, } \]
where \(b:=i+ (k-(j+1))\), \(h:=a+(p-(j+1))\), with \(x\) and \(y\) as in Notation \ref{03 april 2022}.
		
By the obvious cancellations, it follows that
\begin{align*}&\ell_{i,k}^{-2}\,\ell_{i,j}^{2} \,\ell_{a,k}^{2} \,\ell_{j+2,p}^{2} \,\ell_{a,p}^{-2}\\
&= \big(\mathfrak{c}_{a-1,b} \cdot \bar{\mathfrak{c}}_{a-1,b}\cdot \bar{\mathfrak{c}}_{j+1,x+a-1}\big)^{-1}  \cdot \big(\mathfrak{d}_{j+1,k-y+1} \cdot  \mathfrak{d}_{k+1,h} \cdot \bar{\mathfrak{d}}_{k+1,h}\big)^{-1}\\
&= \big( \bar{\mathfrak{c}}_{j+1,x+a-1}\big)^{-1}\cdot \big(\mathfrak{c}_{a-1,b} \cdot \bar{\mathfrak{c}}_{a-1,b} \big)^{-1}  \cdot \big(  \mathfrak{d}_{k+1,h} \cdot \bar{\mathfrak{d}}_{k+1,h}\big)^{-1}\cdot  \big(\mathfrak{d}_{j+1,k-y+1}\big)^{-1}\\
&= \big( \bar{\mathfrak{c}}_{j+1,x+a-1}\big)^{-1}  \cdot \big(  \mathfrak{d}_{k+1,h} \cdot \bar{\mathfrak{d}}_{k+1,h}\big)^{-1}\cdot \big(\mathfrak{c}_{a-1,b} \cdot \bar{\mathfrak{c}}_{a-1,b} \big)^{-1} \cdot \big(\mathfrak{d}_{j+1,k-y+1}\big)^{-1} \tag{middle terms commute}\\
&=\big(\bar{\mathfrak{d}}_{j+1,k-y+1}\big)^{-1}  \cdot \big( \mathfrak{d}_{k+1,h} \cdot \bar{\mathfrak{d}}_{k+1,h}\big)^{-1} \cdot  \big(\mathfrak{c}_{a-1,b} \cdot \bar{\mathfrak{c}}_{a-1,b} \big)^{-1} \cdot \big(\mathfrak{c}_{j+1,x+a-1}\big)^{-1} \tag{Corollary \ref{02 of 22 june 2022}}\\
&=  \big(   \mathfrak{d}_{k+1,h} \cdot \bar{\mathfrak{d}}_{k+1,h}\cdot \bar{\mathfrak{d}}_{j+1,k-y+1}\big)^{-1} \cdot  \big(\mathfrak{c}_{j+1,x+a-1} \cdot \mathfrak{c}_{a-1,b} \cdot \bar{\mathfrak{c}}_{a-1,b} \big)^{-1} 
\end{align*}
By backward substitution, \(\ell_{i,k}^{-2}\,\ell_{i,j}^{2}\,\ell_{a,k}^{2}\,\ell_{j+2,p}^{2}\,\ell_{a,p}^{-2}\)   is then equal to
		
\begin{equation}\label{eqn 4} \scriptstyle{\big((s_{a}\hdots s_{k})^{y} \cdot  \mathfrak{d}_{k+1,h} \cdot \bar{\mathfrak{d}}_{k+1,h}\cdot \bar{\mathfrak{d}}_{j+1,k-y+1}\big)^{-1} \cdot (s_{a}\hdots s_{k})^{x+y} \cdot  \big( \mathfrak{c}_{j+1,a+x-1} \cdot \mathfrak{c}_{a-1,b} \cdot \bar{\mathfrak{c}}_{a-1,b}\cdot (s_{a}\hdots s_{k})^{x}  \big)^{-1} }
\end{equation}
By  Lemma \ref{01 of 22 june 2022}, \eqref{eqn 4} is equal to 
\[
\big((s_{a}\hdots s^{2}_{p}\hdots s_{a})\hdots (s_{j+1}\hdots s^{2}_{p}\hdots s_{j+1})\big)^{-1}\ell_{a,k}^{2}\big((s_{j+1}\hdots s^{2}_{i}\hdots s_{j+1})\hdots (s_{k}\hdots s^{2}_{i}\hdots s_{k})\big)^{-1},
\]
which by further backward substitution  equals  \(\ell_{a,p}^{-2}\,\ell_{j+2,p}^{2}\,\ell_{a,k}^{2}\,\ell_{i,j}^{2}\,\ell_{i,k}^{-2},\) as required.		
\end{proof}

\subsection{Proof of all relations}\label{sectionb}
Set  \(G \coloneqq \langle A_{i,j} \mid R_1 \rangle\) and \(H \coloneqq \langle \gen_{i,j}^{\phantom 2} \mid R_2 \rangle\), where \(R_1\) are the relations in  \textsection \ref{classical presentation} and \( R_2\) are the commutator and box relations in Proposition \ref{3 of 29 january 2022} (substituting $\gen_{i,j}^{\phantom 2}=\ell_{i,j}^{2}$). The following is our main technical lemma.
\begin{lemma}\label{lem:well defined}
There is a well defined group homomorphism \(\upphi \colon G \to H\) defined by
\[
A_{i,j} \mapsto \gen_{i,j-2}^{-1}\, \gen_{i,j-1}^{\phantom 2} \,   \gen_{i+1,j-2}^{\phantom 2} \, \gen_{i+1,j-1}^{-1},
\]
where as in Proposition~\textnormal{\ref{02 of 15 may 2022}} we adopt the convention that $\gen_{i,j}=1$ if $j<i$.
\end{lemma}
\begin{proof}
We show that \(\upphi\) is well defined by tracking the relations \(R_1\) to \(H\).
 
\medskip
\noindent
(1) \(A_{i,j}A_{r,s}= A_{r,s}A_{i,j} ~ \text{if}~ r < s < i <j\). The left hand side is sent to 
\begin{equation}\label{01 of 05 may 2022}	
(\gen_{i,j-2}^{-1}\,\gen_{i,j-1}^{\phantom 2} \,  \gen_{i+1,j-2}^{\phantom 2} \,\gen_{i+1,j-1}^{-1}) (\gen_{r,s-2}^{-1}\,\gen_{r,s-1}^{\phantom 2} \,  \gen_{r+1,s-2}^{\phantom 2} \,\gen_{r+1,s-1}^{-1})		
\end{equation}
and the right hand side is sent to 
\begin{equation}\label{02 of 05 may 2022}
(\gen_{r,s-2}^{-1}\,\gen_{r,s-1}^{\phantom 2} \,  \gen_{r+1,s-2}^{\phantom 2} \,\gen_{r+1,s-1}^{-1}) (\gen_{i,j-2}^{-1}\,\gen_{i,j-1}^{\phantom 2} \,  \gen_{i+1,j-2}^{\phantom 2} \,\gen_{i+1,j-1}^{-1}).
\end{equation}
Each term in the left bracket of  \eqref{01 of 05 may 2022} commutes  with each term  in the right bracket,  by far-away commutativity. Hence  \eqref{01 of 05 may 2022} equals \eqref{02 of 05 may 2022}.
	
\medskip
\noindent
(2)  \(A_{i,j}A_{r,s}= A_{r,s}A_{i,j} ~ \text{if}~ i < r < s <j\). The left hand side is still sent to \eqref{01 of 05 may 2022}, and the right hand side to \eqref{02 of 05 may 2022}. Now each term in the left bracket commutes  with each term in the right bracket by inclusion commutativity, so  \eqref{01 of 05 may 2022} equals \eqref{02 of 05 may 2022}.
	
\medskip
\noindent
(3) \(A^{-1}_{r,s}A_{i,j}A_{r,s}= A_{r,j}A_{i,j}A_{r,j}^{-1} ~ \text{if}~ r <  s =i <j\). The left hand side is sent to 
\[
\scriptstyle (\gen_{r+1,s-1}^{\phantom 2}\,\gen_{r+1,s-2}^{-1} \,  \gen_{r,s-1}^{-1} \,\gen_{r,s-2}^{\phantom 2}) (\gen_{s,j-2}^{-1}\,\gen_{s,j-1}^{\phantom 2} \,  \gen_{s+1,j-2}^{\phantom 2} \,\gen_{s+1,j-1}^{-1})\\ (\gen_{r,s-2}^{-1}\,\gen_{r,s-1}^{\phantom 2} \,  \gen_{r+1,s-2}^{\phantom 2} \,\gen_{r+1,s-1}^{-1}).  
\]
By inclusion commutativity \( [\gen_{r+1,s-2}^{\phantom 2}, \gen_{r,s-1}^{\phantom 2}]=1  \), and by far-away commutativity both \( \gen_{r+1,s-2}^{\phantom 2}\) and
\(\gen_{r,s-2}^{\phantom 2}\) commute with each term in the middle bracket, so the above simplifies to 
\begin{equation}\label{05 of 05 may 2022}
( \gen_{r+1,s-1}  \,\gen_{r,s-1}^{-1}  ) (\gen_{s,j-2}^{-1}\,\gen_{s,j-1} \,  \gen_{s+1,j-2} \,\gen_{s+1,j-1}^{-1})  (    \gen_{r,s-1}\, \gen_{r+1,s-1}^{-1}).
\end{equation}
In a similar way, using \( [\gen_{r+1,j-2}, \gen_{r,j-1}]=1  \)  and  \( \gen_{r,j-1}^{\phantom 2}, \gen_{r+1,j-1}^{\phantom 2}\) commute with each term in the appropriate middle bracket, the right hand side of the relation gets sent to
\begin{equation}\label{06 of 05 may 2022}
(\gen_{r,j-2}^{-1}\,   \gen_{r+1,j-2}^{\phantom 2}) (\gen_{s,j-2}^{-1}\,\gen_{s,j-1}^{\phantom 2} \,  \gen_{s+1,j-2}^{\phantom 2} \,\gen_{s+1,j-1}^{-1})  (  \gen_{r+1,j-2}^{-1}\, \gen_{r,j-2}^{\phantom 2}).  
\end{equation}
By assumption \(r <  s =i <j\), so inclusion commutativity gives \( [\gen_{r,j-2}^{\phantom 2} , \gen_{r+1,s-1}^{-1}]=1 \).  Thus left multiplying both \eqref{05 of 05 may 2022} and \eqref{06 of 05 may 2022} by \(\gen_{r,j-2}^{\phantom 2} \,\gen_{r+1,s-1}^{-1}\), and using other inclusion commutativity, it suffices to prove that
\begin{align*} 
&\phantom{=}(\gen_{r,j-2}^{\phantom 2} \,\gen_{r,s-1}^{-1}  \,\gen_{s,j-2}^{-1} \,\gen_{s+1,j-1}^{-1} \,\gen_{s,j-1}^{\phantom 2} )\,\gen_{s+1,j-2}^{\phantom 2} \,   \gen_{r,s-1}^{\phantom 2}\, \gen_{r+1,s-1}^{-1} \\
&=(\gen_{r+1,j-2}^{\phantom 2} \,\gen_{r+1,s-1}^{-1} \,\gen_{s,j-2}^{-1} \,\gen_{s+1,j-1}^{-1} \,\gen_{s,j-1}^{\phantom 2} ) \,  \gen_{s+1,j-2}^{\phantom 2}\,  \gen_{r+1,j-2}^{-1}\, \gen_{r,j-2}^{\phantom 2}. 
\end{align*}
By applying the box relations to each, this is equivalent to asking that
\begin{align*}
&\phantom{=} ( \gen_{s,j-1}^{\phantom 2} \,\gen_{s+1,j-1}^{-1} \,\gen_{s,j-2}^{-1} \, \gen_{r,s-1}^{-1} \,\gen_{r,j-2}^{\phantom 2}    )\, \gen_{s+1,j-2}^{\phantom 2} \,   \gen_{r,s-1}^{\phantom 2}\, \gen_{r+1,s-1}^{-1}\\
&=(\gen_{s,j-1}^{\phantom 2}  \,   \gen_{s+1,j-1}^{-1} \, \gen_{s,j-2}^{-1} \,\gen_{r+1,s-1}^{-1} \,\gen_{r+1,j-2}^{\phantom 2}   )\,   \gen_{s+1,j-2}^{\phantom 2}\,  \gen_{r+1,j-2}^{-1}\, \gen_{r,j-2}^{\phantom 2}.
\end{align*}
Using \(r <  s =i <j\) and commutativity, these expressions are indeed equal.
	
\medskip
\noindent
(4) \(A^{-1}_{r,s}A_{i,j}A_{r,s}=(A_{i,j}A_{s,j}) A_{i,j} (A_{i,j}A_{s,j})^{-1} ~ \text{if}~ r=i < s <j\). The LHS is sent to 
\begin{equation}\label{11 of 05 may 2022}
{\scriptstyle (\gen_{r+1,s-1}^{\phantom 2} \gen_{r+1,s-2}^{-1} \,  \gen_{r,s-1}^{-1} \,\gen_{r,s-2}^{\phantom 2}) (\gen_{r,j-2}^{-1}\,\gen_{r,j-1}^{\phantom 2} \,  \gen_{r+1,j-2}^{\phantom 2} \,\gen_{r+1,j-1}^{-1}) (\gen_{r,s-2}^{-1}\,\gen_{r,s-1}^{\phantom 2} \,  \gen_{r+1,s-2}^{\phantom 2} \,\gen_{r+1,s-1}^{-1})   }
\end{equation}
while the right hand side is sent to 
\begin{equation}\label{12 of 05 may 2022} 
\begin{aligned}
(\gen_{r,j-2}^{-1}\,\gen_{r,j-1}^{\phantom 2} \,  \gen_{r+1,j-2}^{\phantom 2} \,\gen_{r+1,j-1}^{-1}) (\gen_{s,j-2}^{-1}\,\gen_{s,j-1}^{\phantom 2} \,  \gen_{s+1,j-2}^{\phantom 2} \,\gen_{s+1,j-1}^{-1}) \,(\gen_{r,j-2}^{-1}\,\gen_{r,j-1}^{\phantom 2} \,   \gen_{r+1,j-2}^{\phantom 2} \, &\\ \gen_{r+1,j-1}^{-1})   (\gen_{s+1,j-1}^{\phantom 2}\,\gen_{s+1,j-2}^{-1} \,  \gen_{s,j-1}^{-1} \,\gen_{s,j-2}^{\phantom 2}) (\gen_{r+1,j-1}^{\phantom 2}\,\gen_{r+1,j-2}^{-1} \,  \gen_{r,j-1}^{-1} \,\gen_{r,j-2}^{\phantom 2}).
\end{aligned}
\end{equation}
Since \( r=i < s <j\), we have that \(\gen_{r,j-1}^{\phantom 2}\)  commutes all through and so \(\gen_{r,j-1}^{\phantom 2}\) and \(\gen_{r,j-1}^{-1}\)  in \eqref{12 of 05 may 2022} cancel. Further, left multiply both \eqref{11 of 05 may 2022} and \eqref{12 of 05 may 2022} by \(\gen_{r,j-1}^{-1}\), we can remove all \(\gen_{r,j-1}^{\phantom 2}\) from \eqref{11 of 05 may 2022} and \eqref{12 of 05 may 2022}.
	
\medskip
\noindent
By inclusion commutativity \( \gen_{r+1,s-2}^{\phantom 2}\) commutes with  \( \gen_{r,s-1}^{\phantom 2} , \gen_{r,s-2}^{-1} \) and each term in the middle bracket of \eqref{11 of 05 may 2022} and so \( \gen_{r+1,s-2}^{\phantom 2}\) and \( \gen_{r+1,s-2}^{-1}\) cancel in \eqref{11 of 05 may 2022} . Moreover \(\gen_{r,j-2}^{\phantom 2}\) commutes with each term in the left bracket of \eqref{11 of 05 may 2022}, so \(\gen_{r,j-2}^{-1}\) can be brought to the front of \eqref{11 of 05 may 2022}. Similarly  \(\gen_{r+1,j-1}^{\phantom 2}\) commutes with each term in the second last bracket of \eqref{12 of 05 may 2022} so cancels with  the \(\gen_{r+1,j-1}^{-1}\). Similarly \(\gen_{s+1,j-2}^{-1}\) and \(\gen_{s+1,j-2}^{\phantom 2}\)  cancel in \eqref{12 of 05 may 2022}.  Thus  left multiplying   \eqref{11 of 05 may 2022} and  \eqref{12 of 05 may 2022}  by \(\gen_{r,j-2}^{\phantom 2}\), it suffices to show that 
\begin{align*}\label{13 of 05 may 2022}  
&\phantom{=}( \gen_{r+1,s-1}^{\phantom 2}\,\gen_{r,s-1}^{-1}  \,\gen_{r,s-2}^{\phantom 2}  ) (  \gen_{r+1,j-2}^{\phantom 2} \,\gen_{r+1,j-1}^{-1}) \,(\gen_{r,s-2}^{-1}\,\gen_{r,s-1}^{\phantom 2} \,\gen_{r+1,s-1}^{-1}) \\
&= {\scriptstyle(  \gen_{r+1,j-2}^{\phantom 2} \,\gen_{r+1,j-1}^{-1}) (\gen_{s,j-2}^{-1}\,\gen_{s,j-1}^{\phantom 2} \,  \gen_{s+1,j-1}^{-1} ) \,(\gen_{r,j-2}^{-1}\,  \gen_{r+1,j-2}^{\phantom 2} )   (\gen_{s+1,j-1}^{\phantom 2}\,   \gen_{s,j-1}^{-1} \, \gen_{s,j-2}^{\phantom 2}) ( \gen_{r+1,j-2}^{-1}  \,\gen_{r,j-2}^{\phantom 2}).}	
\end{align*}
By inclusion commutativity, \(\gen_{s,j-2}^{-1}\) commutes with \(\gen_{r+1,j-2}^{\phantom 2}\), \(\gen_{r+1,j-1}^{-1}\), \(\gen_{r+1,j-2}^{-1} , \gen_{r,j-2}^{\phantom 2}\) and further \([\gen_{r,s-1}^{\phantom 2},\gen_{r,s-2}^{-1}]=1\).  By conjugating the above by \(\gen_{r+1,j-2}^{-1}\,\gen_{s,j-2}^{\phantom 2}\), the claim becomes that
\begin{align*}
&( \gen_{r+1,j-2}^{-1}\,\gen_{s,j-2}^{\phantom 2}\,\gen_{r+1,s-1}^{\phantom 2}  \,\gen_{r,s-2}^{\phantom 2} \,\gen_{r,s-1}^{-1})  \gen_{r+1,j-2}^{\phantom 2} \,\gen_{r+1,j-1}^{-1}  (\gen_{r,s-1}^{\phantom 2} \,   \gen_{r,s-2}^{-1}\,\gen_{r+1,s-1}^{-1}\,  \gen_{s,j-2}^{-1} \,\gen_{r+1,j-2}^{\phantom 2})\\
&=\gen_{r+1,j-1}^{-1} \,\gen_{s,j-1}^{\phantom 2} \,  \gen_{s+1,j-1}^{-1} \,\gen_{r,j-2}^{-1}\,  \gen_{r+1,j-2}^{\phantom 2} \, \gen_{s+1,j-1}^{\phantom 2} \, \gen_{s,j-1}^{-1} \,\gen_{r,j-2}^{\phantom 2}
\end{align*}	
By use of the box relations, the top line becomes 
\[
( \gen_{r,s-1}^{-1} \,\gen_{r,s-2}^{\phantom 2} \,\gen_{r+1,s-1}^{\phantom 2}\,\gen_{s,j-2}^{\phantom 2}\, \gen_{r+1,j-2}^{-1}   )  \gen_{r+1,j-2}^{\phantom 2} \,\gen_{r+1,j-1}^{-1} (\gen_{r+1,j-2}^{\phantom 2} \,  \gen_{s,j-2}^{-1} \, \gen_{r+1,s-1}^{-1} \, \gen_{r,s-2}^{-1}\,\gen_{r,s-1}^{\phantom 2})
\]
which by inclusion commutativity  and obvious cancellations simplifies to
\[
\gen_{r,s-1}^{-1} \,\gen_{r,s-2}^{\phantom 2}  \,  \gen_{r+1,j-1}^{-1} \,\gen_{r+1,j-2}^{\phantom 2}  \, \gen_{r,s-2}^{-1}\,\gen_{r,s-1}^{\phantom 2}.      
\]
By right multiplying  by \( \gen_{r,s-2}^{\phantom 2}\) and left multiplying by \( \gen_{r,s-1}^{\phantom 2}\), the claim becomes that
\begin{align*}
&\phantom{=} \gen_{r,s-2}^{\phantom 2}   \,\gen_{r+1,j-1}^{-1} \, \gen_{r+1,j-2}^{\phantom 2}  \,  \gen_{r,s-1}^{\phantom 2} \\
&=\gen_{r,s-1}^{\phantom 2}\,\gen_{s+1,j-1}^{-1}( \gen_{r+1,j-1}^{-1} \, \gen_{s,j-1}^{\phantom 2} \,  \gen_{r+1,j-2}^{\phantom 2}\, \gen_{r,s-2}^{\phantom 2} \,\gen_{r,j-2}^{-1})\gen_{s+1,j-1}^{\phantom 2} \, \gen_{s,j-1}^{-1}\, \gen_{r,j-2}^{\phantom 2}.
\end{align*}	
By the box relations, the bottom line equals 
\[ 
\gen_{r,s-1}^{\phantom 2}\,\gen_{s+1,j-1}^{-1}(\gen_{r,j-2}^{-1} \, \gen_{r,s-2}^{\phantom 2} \,\gen_{r+1,j-2}^{\phantom 2} \, \gen_{s,j-1}^{\phantom 2} \,\gen_{r+1,j-1}^{-1}  )\gen_{s+1,j-1}^{\phantom 2} \, \gen_{s,j-1}^{-1} \, \gen_{r,j-2}^{\phantom 2}
\]
which simplifies to
\[
\gen_{r,s-2}^{\phantom2}\,\gen_{s+1,j-1}^{-1}(\gen_{r,j-2}^{-1}\,\gen_{r,s-1}^{\phantom 2} \,\gen_{r+1,j-2}^{\phantom 2} \,\gen_{s+1,j-1}^{\phantom 2} \, \gen_{r+1,j-1}^{-1}    )    \gen_{r,j-2}^{\phantom 2}.
\] 
One final application of the box relations, and commutativity, proves the claim.

\medskip
\noindent	
(5)  \(A^{-1}_{r,s}A_{i,j}A_{r,s}=(A_{r,j}A_{s,j}A_{r,j}^{-1}A_{s,j}^{-1}) A_{i,j} (A_{r,j}A_{s,j}A_{r,j}^{-1}A_{s,j}^{-1})^{-1} ~ \text{if}~ r<i < s <j\).
	
The factor  \((A_{r,j}A_{s,j}A_{r,j}^{-1}A_{s,j}^{-1})\)  gets sent  to 
\begin{align*}
(\gen_{r,j-2}^{-1}\,\gen_{r,j-1}^{\phantom 2} \,  \gen_{r+1,j-2}^{\phantom 2} \,\gen_{r+1,j-1}^{-1}) (\gen_{s,j-2}^{-1}\,\gen_{s,j-1}^{\phantom 2} \,  \gen_{s+1,j-2}^{\phantom 2} \,\gen_{s+1,j-1}^{-1}) (\gen_{r+1,j-1}^{\phantom 2}\,&\\ \gen_{r+1,j-2}^{-1} \,  \gen_{r,j-1}^{-1} \,\gen_{r,j-2}^{\phantom 2})   (\gen_{s+1,j-1}^{\phantom 2}\,\gen_{s+1,j-2}^{-1} \,  \gen_{s,j-1}^{-1} \,\gen_{s,j-2}^{\phantom 2}),
\end{align*} 
which by  inclusion commutativity and cancellation equals 
\begin{equation}\label{24 of 05 may 2022}
(\gen_{r,j-2}^{-1}\,  \gen_{r+1,j-2}^{\phantom 2} ) (\gen_{s,j-2}^{-1}\,\gen_{s,j-1}^{\phantom 2}  \,\gen_{s+1,j-1}^{-1}) ( \gen_{r+1,j-2}^{-1}  \,\gen_{r,j-2}^{\phantom 2})   (\gen_{s+1,j-1}^{\phantom 2}  \, \gen_{s,j-1}^{-1} \,\gen_{s,j-2}^{\phantom 2}).
\end{equation}
Now the left hand side of the relation is sent to 
\[
{\scriptstyle (\gen_{r+1,s-1}^{\phantom 2}\,\gen_{r+1,s-2}^{-1} \,  \gen_{r,s-1}^{-1} \,\gen_{r,s-2}^{\phantom 2}) (\gen_{i,j-2}^{-1}\,\gen_{i,j-1}^{\phantom 2} \,  \gen_{i+1,j-2}^{\phantom 2} \,\gen_{i+1,j-1}^{-1})  (\gen_{r,s-2}^{-1}\,\gen_{r,s-1}^{\phantom 2} \,  \gen_{r+1,s-2}^{\phantom 2} \,\gen_{r+1,s-1}^{-1})} 
\]
whilst using \eqref{24 of 05 may 2022} the right hand side of the relation is sent to
\[
\begin{aligned}
(\gen_{r,j-2}^{-1}\,  \gen_{r+1,j-2}^{\phantom 2} ) (\gen_{s,j-2}^{-1}\,\gen_{s,j-1}^{\phantom 2}  \,\gen_{s+1,j-1}^{-1}) ( \gen_{r+1,j-2}^{-1}  \,\gen_{r,j-2}^{\phantom 2})   (\gen_{s+1,j-1}^{\phantom 2}  \,   \gen_{s,j-1}^{-1} \,\gen_{s,j-2}^{\phantom 2})  &\\(\gen_{i,j-2}^{-1}\,\gen_{i,j-1}^{\phantom 2} \,  \gen_{i+1,j-2}^{\phantom 2} \,\gen_{i+1,j-1}^{-1})( \gen_{s,j-2}^{-1} \,   \gen_{s,j-1}^{\phantom 2} \,\gen_{s+1,j-1}^{-1} )(  \gen_{r,j-2}^{-1} \,\gen_{r+1,j-2}^{\phantom 2}) &\\(\gen_{s+1,j-1}^{\phantom 2} \,\gen_{s,j-1}^{-1}  \,\gen_{s,j-2}^{\phantom 2} ) ( \gen_{r+1,j-2}^{-1} \,\gen_{r,j-2}^{\phantom 2} ).
\end{aligned}
\]
Multiplying both to the left and right of the last two equations by \(\gen_{r+1,s-1}^{-1}\) and  using  inclusion commutativity to commute \(\gen_{s,j-2}^{\phantom 2} \) through the middle bracket of the right hand side of the relation  , it suffices to prove that
\begin{align}\label{32 of 05 may 2022}
& \gen_{r+1,s-2}^{-1} \,\gen_{r,s-1}^{-1} \,\gen_{r,s-2}^{\phantom 2} \, \gen_{i,j-2}^{-1}\,\gen_{i,j-1}^{\phantom 2} \,  \gen_{i+1,j-2}^{\phantom 2} \,\gen_{i+1,j-1}^{-1} \,\gen_{r,s-2}^{-1}\,\gen_{r,s-1}^{\phantom 2} \,  \gen_{r+1,s-2}^{\phantom 2}\nonumber\\
&=\gen_{r+1,s-2}^{-1}\,\gen_{r,s-1}^{-1} \, \gen_{r,s-2}^{\phantom 2} ( \gen_{i,j-1}^{\phantom 2} \, \gen_{i,j-2}^{-1} \,   \gen_{i+1,j-2}^{\phantom 2} \, \gen_{i+1,j-1}^{-1})  \gen_{r,s-2}^{-1} \,\gen_{r,s-1}^{\phantom 2} \, \gen_{r+1,s-2}^{\phantom 2} 
\end{align}
equals 
\[
\begin{aligned}
\gen_{r,j-2}^{-1} (\gen_{r+1,j-2}^{\phantom 2} \,\gen_{r+1,s-1}^{-1}  \,    \gen_{s,j-2}^{-1} \,\gen_{s+1,j-1}^{-1} \,\gen_{s,j-1}^{\phantom 2}  ) \gen_{r+1,j-2}^{-1}  \,\gen_{r,j-2}^{\phantom 2}\,\gen_{s+1,j-1}^{\phantom 2}  \,   \gen_{s,j-1}^{-1} \, \gen_{i,j-2}^{-1}\,&\\ \gen_{i,j-1}^{\phantom 2} \,  \gen_{i+1,j-2}^{\phantom 2} \,\gen_{i+1,j-1}^{-1}\,\gen_{s,j-1}^{\phantom 2} \,\gen_{s+1,j-1}^{-1} \, \gen_{r,j-2}^{-1} \,\gen_{r+1,j-2}^{\phantom 2}  ( \gen_{s,j-1}^{-1}  \,\gen_{s+1,j-1}^{\phantom 2} \,&\\ \gen_{s,j-2}^{\phantom 2} \,\gen_{r+1,s-1}^{\phantom 2}  \gen_{r+1,j-2}^{-1} ) \gen_{r,j-2}^{\phantom 2} .
\end{aligned}
\]
By the  box relations, the last equation equals
\begin{align*}
\gen_{r,j-2}^{-1} (\gen_{s,j-1}^{\phantom 2}\,\gen_{s+1,j-1}^{-1} \, \gen_{s,j-2}^{-1} \,\gen_{r+1,s-1}^{-1}  \, \gen_{r+1,j-2}^{\phantom 2}) \gen_{r+1,j-2}^{-1}  \,\gen_{r,j-2}^{\phantom 2} \,\gen_{s+1,j-1}^{\phantom 2}  \,   \gen_{s,j-1}^{-1} \,\gen_{i,j-2}^{-1}\, &\\ \gen_{i,j-1}^{\phantom 2} \,  \gen_{i+1,j-2}^{\phantom 2} \,\gen_{i+1,j-1}^{-1} \, \gen_{s,j-1}^{\phantom 2} \,\gen_{s+1,j-1}^{-1} \,  \gen_{r,j-2}^{-1} \,\gen_{r+1,j-2}^{\phantom 2} (\gen_{r+1,j-2}^{-1} \,\gen_{r+1,s-1}^{\phantom 2} \, &\\ \gen_{s,j-2}^{\phantom 2} \,\gen_{s+1,j-1}^{\phantom 2} \, \gen_{s,j-1}^{-1}  ) \gen_{r,j-2}^{\phantom 2} ,
\end{align*}
which by obvious cancellations  simplifies to
\[
\begin{aligned}
\gen_{r,j-2}^{-1} \,\gen_{s,j-1}^{\phantom 2}\,\gen_{s+1,j-1}^{-1} \, \gen_{s,j-2}^{-1} \,\gen_{r+1,s-1}^{-1} \, \gen_{r,j-2}^{\phantom 2} \, \gen_{s+1,j-1}^{\phantom 2}  \,   \gen_{s,j-1}^{-1} \,\gen_{i,j-2}^{-1}\,\gen_{i,j-1}^{\phantom 2} \,   \gen_{i+1,j-2}^{\phantom 2} \,&\\  \gen_{i+1,j-1}^{-1} \, \gen_{s,j-1}^{\phantom 2} \,\gen_{s+1,j-1}^{-1} \,  \gen_{r,j-2}^{-1} \, \gen_{r+1,s-1}^{\phantom 2} \, \gen_{s,j-2}^{\phantom 2} \,\gen_{s+1,j-1}^{\phantom 2} \, \gen_{s,j-1}^{-1}  \, \gen_{r,j-2}^{\phantom 2}.
\end{aligned}
\]
Inserting  \(\gen_{r,s-1}^{\phantom 2}\,\gen_{r,s-1}^{-1}=1\)  twice  and using \([\gen_{r,s-1}^{\phantom 2}, \gen_{r,j-2}^{\phantom 2}] =1\), this equals
\[
\begin{split}
\gen_{r,j-2}^{-1} (\gen_{s,j-1}^{\phantom 2}\,\gen_{s+1,j-1}^{-1} \, \gen_{s,j-2}^{-1} \,\gen_{r,s-1}^{-1} \, \gen_{r,j-2}^{\phantom 2}) \gen_{r,s-1}^{\phantom 2} \,\gen_{s+1,j-1}^{\phantom 2}  \, \gen_{r+1,s-1}^{-1} \,  \gen_{s,j-1}^{-1} \,  \gen_{i,j-2}^{-1}\, &\\ \gen_{i,j-1}^{\phantom 2} \,   \gen_{i+1,j-2}^{\phantom 2}  \, \gen_{i+1,j-1}^{-1} \, \gen_{s,j-1}^{\phantom 2} \, \gen_{r+1,s-1}^{\phantom 2} \,  \gen_{s+1,j-1}^{-1} \,   \gen_{r,s-1}^{-1}( \gen_{r,j-2}^{-1} \gen_{r,s-1}^{\phantom 2}&\\ \gen_{s,j-2}^{\phantom 2} \,\gen_{s+1,j-1}^{\phantom 2} \, \gen_{s,j-1}^{-1}  ) \gen_{r,j-2}^{\phantom 2} ,
\end{split}
\]
which again by the box relations becomes 
\[
\begin{aligned}
\gen_{r,j-2}^{-1} (\gen_{r,j-2}^{\phantom 2} \gen_{r,s-1}^{-1} \,\gen_{s,j-2}^{-1} \,\gen_{s+1,j-1}^{-1} \cdot\gen_{s,j-1}^{\phantom 2}    )\gen_{r,s-1}^{\phantom 2}  \,\gen_{s+1,j-1}^{\phantom 2}  \,  \gen_{r+1,s-1}^{-1} \,\gen_{s,j-1}^{-1}\,  \gen_{i,j-2}^{-1}\,  &\\ \gen_{i,j-1}^{\phantom 2} \,   \gen_{i+1,j-2}^{\phantom 2}  \,  \gen_{i+1,j-1}^{-1} \,\gen_{s,j-1}^{\phantom 2}  \, \gen_{r+1,s-1}^{\phantom 2}\,\gen_{s+1,j-1}^{-1} \, \gen_{r,s-1}^{-1}(\gen_{s,j-1}^{-1}\,\gen_{s+1,j-1}^{\phantom 2} \, &\\ \gen_{s,j-2}^{\phantom 2}  \,\gen_{r,s-1}^{\phantom 2} \,\gen_{r,j-2}^{-1}  )  \gen_{r,j-2}^{\phantom 2}.
\end{aligned}
\]
By commutativity, this simplifies to
\begin{equation}\label{33 of 05 may 2022}
\begin{split}
\gen_{r,s-1}^{-1} \,   \gen_{s,j-2}^{-1} \,   \gen_{s,j-1}^{\phantom 2} \gen_{r,s-1}^{\phantom 2} \,  \gen_{r+1,s-1}^{-1} \, \gen_{s,j-1}^{-1}\, \gen_{i,j-2}^{-1}\, \gen_{i,j-1}^{\phantom 2} \,   \gen_{i+1,j-2}^{\phantom 2}  \, \gen_{i+1,j-1}^{-1} \,  \gen_{s,j-1}^{\phantom 2} \, \,\gen_{r+1,s-1}^{\phantom 2} \,&\\  \gen_{r,s-1}^{-1} \, \gen_{s,j-1}^{-1}  \, \gen_{s,j-2}^{\phantom 2}  \,   \gen_{r,s-1}^{\phantom 2}.
\end{split}
\end{equation}
Now, conjugate   both \eqref{32 of 05 may 2022}  and \eqref{33 of 05 may 2022}  by \(\gen_{r,s-1}^{\phantom 2}\),  and use that \([\gen_{r+1,s-2}^{\phantom 2}, \gen_{r,s-1}^{\phantom 2} ]=1 \). Further, left   multiply  both \eqref{32 of 05 may 2022}  and \eqref{33 of 05 may 2022}  by \(\gen_{s,j-2}^{\phantom 2}\),  and   use the fact that \(\gen_{s,j-2}^{\phantom 2}\) commutes all through \eqref{32 of 05 may 2022}.  Then, right multiply both \eqref{32 of 05 may 2022}  and \eqref{33 of 05 may 2022}  by  \(\gen_{s,j-2}^{-1}\). Furthermore, conjugate  both \eqref{32 of 05 may 2022}  and \eqref{33 of 05 may 2022}  by \(\gen_{s,j-1}^{-1}\) and using commutativity,   it suffices to show that 
\begin{align*} 
&\gen_{r+1,s-2}^{-1} \, \gen_{r,s-2}^{\phantom 2} ( \gen_{i,j-1}^{\phantom 2} \, \gen_{s,j-1}^{-1} \, \gen_{i,j-2}^{-1} \,   \gen_{i+1,j-2}^{\phantom 2} \,\gen_{s,j-1}^{\phantom 2} \,\gen_{i+1,j-1}^{-1})  \gen_{r,s-2}^{-1} \, \gen_{r+1,s-2}^{\phantom 2}
\end{align*}
equals 
\begin{align}\label{01 of 07 may 2022}
\gen_{r,s-1}^{\phantom 2} \,  \gen_{r+1,s-1}^{-1} \, \gen_{s,j-1}^{-1}\, \gen_{i,j-2}^{-1}\, \gen_{i,j-1}^{\phantom 2} \,   \gen_{i+1,j-2}^{\phantom 2}  \, \gen_{i+1,j-1}^{-1} \,  \gen_{s,j-1}^{\phantom 2} \, \,\gen_{r+1,s-1}^{\phantom 2} \, \gen_{r,s-1}^{-1}.  
\end{align} 
Inserting \(\gen_{r,s-2}^{-1} \,   \gen_{r,j-2}^{\phantom 2}\gen_{r,j-2}^{-1} \, \gen_{r,s-2}^{\phantom 2}=1  \),  the top line of the claim becomes
\begin{align*} 
{\scriptstyle\gen_{r+1,s-2}^{-1} \, \gen_{r,s-2}^{\phantom 2} ( \gen_{i,j-1}^{\phantom 2} \, \gen_{s,j-1}^{-1} \, \gen_{i,j-2}^{-1} \,  \gen_{r,s-2}^{-1} \,   \gen_{r,j-2}^{\phantom 2} ) (\gen_{r,j-2}^{-1} \, \gen_{r,s-2}^{\phantom 2} \,  \gen_{i+1,j-2}^{\phantom 2} \,\gen_{s,j-1}^{\phantom 2} \,\gen_{i+1,j-1}^{-1})  \gen_{r,s-2}^{-1} \, \gen_{r+1,s-2}^{\phantom 2}}
\end{align*}
which by the box relations equals
\begin{equation*}
{\scriptstyle\gen_{r+1,s-2}^{-1} \gen_{r,s-2}^{\phantom 2} ( \gen_{r,j-2}^{\phantom 2}  \,  \gen_{r,s-2}^{-1} \,\gen_{i,j-2}^{-1}\, \gen_{s,j-1}^{-1} \,   \gen_{i,j-1}^{\phantom 2} ) (\gen_{i+1,j-1}^{-1}\, \gen_{s,j-1}^{\phantom 2} \,\gen_{i+1,j-2}^{\phantom 2} \,  \gen_{r,s-2}^{\phantom 2} \,\gen_{r,j-2}^{-1})  \gen_{r,s-2}^{-1}  \gen_{r+1,s-2}^{\phantom 2},}
\end{equation*}
which by inclusion commutativity  and cancellations  simplifies to
\begin{equation}\label{34 of 05 may 2022}
\begin{aligned} \gen_{r+1,s-2}^{-1} \, \gen_{r,j-2}^{\phantom 2}  \,\gen_{i,j-2}^{-1}\,    \gen_{i,j-1}^{\phantom 2} \,\gen_{i+1,j-1}^{-1} \,\gen_{i+1,j-2}^{\phantom 2} \,\gen_{r,j-2}^{-1} \, \gen_{r+1,s-2}^{\phantom 2}.
\end {aligned}
\end{equation}
Inserting  \(\gen_{s+1,j-1}^{\phantom 2}\,\gen_{s+1,j-1}^{-1} =1 \) in \eqref{34 of 05 may 2022}, using commutativity  and conjugating \eqref{34 of 05 may 2022}  and \eqref{01 of 07 may 2022} by \(\gen_{r,s-1}^{-1}\), the claim becomes
\begin{align*}
&\gen_{r+1,s-2}^{-1} (\gen_{r,j-2}^{\phantom 2}  \,\gen_{r,s-1}^{-1}   \,\gen_{i,j-2}^{-1}\,\gen_{s+1,j-1}^{-1} \,   \gen_{i,j-1}^{\phantom 2} ) (\gen_{i+1,j-1}^{-1} \,\gen_{s+1,j-1}^{\phantom 2} \, \gen_{i+1,j-2}^{\phantom 2} \, \gen_{r,s-1}^{\phantom 2} \,  \gen_{r,j-2}^{-1})   \gen_{r+1,s-2}^{\phantom 2}  \\
&=\gen_{r+1,s-1}^{-1} \,  \gen_{s,j-1}^{-1}  \gen_{i,j-2}^{-1}\, \gen_{i,j-1}^{\phantom 2} \,   \gen_{i+1,j-2}^{\phantom 2}  \, \gen_{i+1,j-1}^{-1}  \gen_{s,j-1}^{\phantom 2}  \,  \gen_{r+1,s-1}^{\phantom 2}.
\end{align*}
Using  box relations the top line of the claim equals
\[
\gen_{r+1,s-2}^{-1} (\gen_{i,j-1}^{\phantom 2}  \,\gen_{s+1,j-1}^{-1}   \,\gen_{i,j-2}^{-1}\,\gen_{r,s-1}^{-1} \,   \gen_{r,j-2}^{\phantom 2} ) (\gen_{r,j-2}^{-1} \,\gen_{r,s-1}^{\phantom 2} \, \gen_{i+1,j-2}^{\phantom 2} \,\gen_{s+1,j-1}^{\phantom 2} \, \gen_{i+1,j-1}^{-1})   \gen_{r+1,s-2}^{\phantom 2}, 
\]
which simplifies to
\[
\gen_{r+1,s-2}^{-1} (\gen_{i,j-1}^{\phantom 2}  \,\gen_{s+1,j-1}^{-1}   \,\gen_{i,j-2}^{-1}  \, \gen_{i+1,j-2}^{\phantom 2} \,\gen_{s+1,j-1}^{\phantom 2} \, \gen_{i+1,j-1}^{-1})  \gen_{r+1,s-2}^{\phantom 2},
\]
which after adding  \( \gen_{r+1,s-1}^{-1} \,\gen_{r+1,j-2}^{\phantom 2} \, \gen_{r+1,j-2}^{-1}  \,  \gen_{r+1,s-1}^{\phantom 2} =1\) becomes
\[
{\scriptstyle \gen_{r+1,s-2}^{-1} (\gen_{i,j-1}^{\phantom 2}\, \gen_{s+1,j-1}^{-1} \,\gen_{i,j-2}^{-1} \,  \gen_{r+1,s-1}^{-1} \,\gen_{r+1,j-2}^{\phantom 2} ) (\gen_{r+1,j-2}^{-1}  \,  \gen_{r+1,s-1}^{\phantom 2} \,\gen_{i+1,j-2}^{\phantom 2}\, \gen_{s+1,j-1}^{\phantom 2} \,   \gen_{i+1,j-1}^{-1} )  \gen_{r+1,s-2}^{\phantom 2}.}
\]
Thus by  the  box relations this equals
\[
{\scriptstyle\gen_{r+1,s-2}^{-1} (\gen_{r+1,j-2}^{\phantom 2}\, \gen_{r+1,s-1}^{-1} \,\gen_{i,j-2}^{-1} \,  \gen_{s+1,j-1}^{-1} \,\gen_{i,j-1}^{\phantom 2} ) (\gen_{i+1,j-1}^{-1}  \,  \gen_{s+1,j-1}^{\phantom 2} \,\gen_{i+1,j-2}^{\phantom 2}\, \gen_{r+1,s-1}^{\phantom 2} \,   \gen_{r+1,j-2}^{-1} ) \gen_{r+1,s-2}^{\phantom 2}},
\]
which by   inclusion commutativity  becomes 
\begin{align}\label{02 of 07 may 2022}
\gen_{r+1,s-2}^{-1} \gen_{r+1,j-2}^{\phantom 2}\, \gen_{r+1,s-1}^{-1}\,\gen_{i,j-2}^{-1} \,\gen_{i,j-1}^{\phantom 2} \,\gen_{i+1,j-1}^{-1}    \,\gen_{i+1,j-2}^{\phantom 2}\,\gen_{r+1,s-1}^{\phantom 2}\,  \gen_{r+1,j-2}^{-1}  \gen_{r+1,s-2}^{\phantom 2}. 
\end{align}
Now conjugate   the bottom line of the claim  and \eqref{02 of 07 may 2022} by \(\gen_{r+1,s-1}^{\phantom 2}\)  and use commutativity. We further insert  \(\gen_{s,j-1}^{\phantom 2}\, \gen_{s,j-1}^{- 2} =1\) in   \eqref{02 of 07 may 2022} and use commutativity so that the claim becomes 
\begin{align*}
&(\gen_{r+1,j-2}^{\phantom 2}\,\gen_{r+1,s-2}^{-1} \,   \gen_{i,j-2}^{-1} \, \gen_{s,j-1}^{-1} \, \gen_{i,j-1}^{\phantom 2} \,\gen_{i+1,j-1}^{-1})  \,(\gen_{s,j-1}^{\phantom 2} \,    \gen_{i+1,j-2}^{\phantom 2} \,  \gen_{r+1,s-2}^{\phantom 2} \,   \gen_{r+1,j-2}^{-1}) \\
&=\gen_{s,j-1}^{-1}  \gen_{i,j-2}^{-1}\, \gen_{i,j-1}^{\phantom 2} \,   \gen_{i+1,j-2}^{\phantom 2}  \, \gen_{i+1,j-1}^{-1}  \gen_{s,j-1}^{\phantom 2}. 
\end{align*}
By  use of the box relations on the the top line of the claim, commutativity, and  obvious cancellations, the claim  holds. 
\end{proof}

As a slight abuse of notation, set $\scrA\colonequals \ell_{\scrA}^2$. The following is our main result.
\begin{theorem}\label{Feb 2022}
The pure braid group $\PB_{A_{n}}$  has a presentation with generators given by connected subgraphs $\scrA \subseteq  A_n=\Afour{B}{B}{B}{B}$, subject to the relations
\begin{enumerate}
\item  \( \scrA\cdot\scrB = \scrB\cdot \scrA\) if \(d(\scrA, \scrB)\geq 2\), or \(\scrA \subseteq \scrB\), or \(\scrB\subseteq\scrA\).
\item For all $\scrA$ and all $\scrC$ such that $d(\scrA,\scrC)=2$, then
\[(\scrA\cup\scrB)^{-1}\cdot(\scrA\cdot\scrB\cdot\scrC)\cdot(\scrB\cup\scrC)^{-1}=
(\scrC\cup\scrB)^{-1}\cdot(\scrC\cdot\scrB\cdot\scrA)\cdot(\scrB\cup\scrA)^{-1}
\]
for all   $\scrB$  compatible with \((\scrA, \scrC) \).
\end{enumerate}		
\end{theorem}
\begin{proof}
Consider the homomorphism $\upphi\colon G\to H$ defined in Lemma~\ref{lem:well defined}. By Proposition \ref{3 of 29 january 2022}, there is also a homomomorphism \(H\twoheadrightarrow \PB_{A_{n}} \) sending \(\gen_{i,j} \mapsto \ell_{i,j}^{2},\)  which is surjective by Corollary  \ref{04 of 15 may 2022}. By Proposition~\ref{02 of 15 may 2022}, by chasing the generators $A_{ij}$ in both directions, the following diagram commutes
\[
\begin{tikzpicture}[bend angle=20, looseness=1]
\node (a) at (-2,0) [] {$H$};
\node (b) at (0,0) [] {\(\mathrm{\PB}_{A_{n}}\)};
\node (C2) at (0,2) [] {$G$};
\draw[transform canvas={yshift=0ex},->>] (a) --node[below] {}  (b) ;
\path[->,font=\scriptsize,>=angle 90] ;
\draw[transform canvas={yshift=0ex}, <-] (b) --node[right] {$ \cong $} (C2); 
\draw[,<-] (a) --node[above left] {$\scriptstyle \upphi$} (C2);
\end{tikzpicture}
\]  
and so it follows that $\upphi$ is injective.   We will now prove that $\upphi$ is also surjective, by establishing that every generator $\gen_{i,j}$ of $H$ is in the image of $\upphi$ using induction on $j-i$.

If $j-i=0$ then by convention $\upphi(A_{i,i+1})=\gen_{i,i}$, and so $\gen_{i,i}$ belongs to the image.  Thus we can consider an arbitrary $\gen_{i,j}$ and assume that all $\gen_{a,b}$ are in the image of $\upphi$ whenever $b-a<j-i$.  Now
\[
\upphi(A_{i,j+1})=\gen_{i,j-1}^{-1}\, \gen_{i,j}^{\phantom 2} \,   \gen_{i+1,j-1}^{\phantom 2} \, \gen_{i+1,j}^{-1},
\] 
and by hypothesis there exists $g_1,g_2,g_3\in G$ such that $\upphi(g_1)=\gen_{i,j-1}$, $\upphi(g_2)=\gen_{i+1,j-1}$ and $\upphi(g_3)=\gen_{i+1,j}$.  But then $\upphi(g_1A_{i,j+1}g_3g_2^{-1})=\gen_{i,j}$, and so $\upphi$ is surjective.\end{proof}

A comparison with other presentations in the literature requires us to reindex.  For any pair $ij$ with $1\leq i<j\leq n+1$ we can consider the connected subset $[i,j-1]$ of $A_n$.  This gives a bijection between such pairs, and connected subsets.

In this way, $\PB_{A_{n}}$ is generated by pairs $ij$ such that $1\leq i<j\leq n+1$, and as is standard we draw this as a triangle,  where the case of $A_n$ with $n=8$ is illustrated below. 
\[
\begin{tikzpicture}[scale=1.5]
\node at (0,0) {$\scriptstyle 12$};
\node at (0.5,0) {$\scriptstyle 23$};
\node at (1,0) {$\scriptstyle 34$};
\node at (1.5,0) {$\scriptstyle 45$};
\node at (2,0) {$\scriptstyle 56$};
\node at (2.5,0) {$\scriptstyle 67$};
\node at (3,0) {$\scriptstyle 78$};
\node at (3.5,0) {$\scriptstyle 89$};
%

\node at (0.25,-0.25)  {$\scriptstyle 13$};
\node at (0.75,-0.25)  {$\scriptstyle 24$};
\node at (1.25,-0.25)  {$\scriptstyle 35$};
\node at (1.75,-0.25)  {$\scriptstyle 46$};
\node at (2.25,-0.25)  {$\scriptstyle 57$};
\node at (2.75,-0.25)  {$\scriptstyle 68$};
\node at (3.25,-0.25)  {$\scriptstyle 79$};
\node at (0.5,-0.5) {$\scriptstyle 14$};
\node at (1,-0.5) {$\scriptstyle 25$};
\node at (1.5,-0.5) {$\scriptstyle 36$};
\node at (2,-0.5) {$\scriptstyle 47$};
\node at (2.5,-0.5) {$\scriptstyle 58$};
\node at (3,-0.5) {$\scriptstyle 69$};
\node at (0.75,-0.75)  {$\scriptstyle 15$};
\node at (1.25,-0.75)  {$\scriptstyle 26$};
\node at (1.75,-0.75)  {$\scriptstyle 37$};
\node at (2.25,-0.75)  {$\scriptstyle 48$};
\node at (2.75,-0.75)  {$\scriptstyle 59$};

%
\node at (1,-1) {$\scriptstyle 16$};
\node at (1.5,-1) {$\scriptstyle 27$};
\node at (2,-1) {$\scriptstyle 38$};
\node at (2.5,-1) {$\scriptstyle 49$};
\node at (1.25,-1.25)  {$\scriptstyle 17$};
\node at (1.75,-1.25)  {$\scriptstyle 28$};
\node at (2.25,-1.25)  {$\scriptstyle 39$};
\node at (1.5,-1.5) {$\scriptstyle 18$};
\node at (2,-1.5) {$\scriptstyle 29$};
\node at (1.75,-1.75)  {$\scriptstyle 19$};			
\end{tikzpicture}
\]
In this notation, Theorem \ref{Feb 2022}  translate into the following.	
\begin{cor}\label{07 of 15 may 2022}
The $\PB_{A_{n}}$ is generated by $ \{z_{ij} \mid 1\leq i<j\leq n+1 \}$ subject to the relations
\begin{enumerate}
\item $z_{ij}z_{kl}=z_{kl}z_{ij}$ if $k\geq j+2$, or $k\leq i\leq j \leq l $ or   $l\leq k\leq l \leq j $.
\item  There is an equality \[z_{iy}^{-1}(z_{ij}z_{xy}z_{j+1k})z_{xk}^{-1}=z_{xk}^{-1}(z_{j+1k}z_{xy}z_{ij})z_{iy}^{-1}\]
whenever 	\(1\leq i<x\leq j   <y  < k \leq n+1.\)	
\end{enumerate} 
\end{cor}

\begin{remark}\label{01 of 08 july 2022}
Choose $\scrA=15$, i.e.\ $\scrA=[1,4]$, then the possible $\scrC$ are those on the dotted line.  For the choice $\scrC=69$, the subset $\scrB$ can be any of those in the box illustrated below. 
\[
\begin{tikzpicture}[scale=1.5]
\node at (0,0) {$\scriptstyle 12$};
\node at (0.5,0) {$\scriptstyle 23$};
\node at (1,0) {$\scriptstyle 34$};
\node at (1.5,0) {$\scriptstyle 45$};
\node at (2,0) {$\scriptstyle 56$};
\node at (2.5,0) {$\scriptstyle 67$};
\node at (3,0) {$\scriptstyle 78$};
\node at (3.5,0) {$\scriptstyle 89$};
\draw[densely dotted] (2.5,0) -- (3,-0.5);

\node at (0.25,-0.25)  {$\scriptstyle 13$};
\node at (0.75,-0.25)  {$\scriptstyle 24$};
\node at (1.25,-0.25)  {$\scriptstyle 35$};
\node at (1.75,-0.25)  {$\scriptstyle 46$};
\node at (2.25,-0.25)  {$\scriptstyle 57$};
\node at (2.75,-0.25)  {$\scriptstyle 68$};
\node at (3.25,-0.25)  {$\scriptstyle 79$};
\node at (0.5,-0.5) {$\scriptstyle 14$};
\node at (1,-0.5) {$\scriptstyle 25$};
\node at (1.5,-0.5) {$\scriptstyle 36$};
\node at (2,-0.5) {$\scriptstyle 47$};
\node at (2.5,-0.5) {$\scriptstyle 58$};
\node at (3,-0.5) {$\scriptstyle 69$};
\draw (3,-0.5) circle (4pt);
\draw (0.75,-0.75) circle (4pt);
\node at (0.75,-0.75)  {$\scriptstyle 15$};
\node at (1.25,-0.75)  {$\scriptstyle 26$};
\node at (1.75,-0.75)  {$\scriptstyle 37$};
\node at (2.25,-0.75)  {$\scriptstyle 48$};
\node at (2.75,-0.75)  {$\scriptstyle 59$};

\draw (1,-0.75) -- (1.75,-1.5) -- (2.75,-0.5) -- (2,0.25 ) -- cycle;
\node at (1,-1) {$\scriptstyle 16$};
\node at (1.5,-1) {$\scriptstyle 27$};
\node at (2,-1) {$\scriptstyle 38$};
\node at (2.5,-1) {$\scriptstyle 49$};
\node at (1.25,-1.25)  {$\scriptstyle 17$};
\node at (1.75,-1.25)  {$\scriptstyle 28$};
\node at (2.25,-1.25)  {$\scriptstyle 39$};
\node at (1.5,-1.5) {$\scriptstyle 18$};
\node at (2,-1.5) {$\scriptstyle 29$};
\node at (1.75,-1.75)  {$\scriptstyle 19$};
\end{tikzpicture}		
\]		
This justifies calling the length five relations the box relations, since each relation is characterized by a choice of an element in the box.	
\end{remark}
\begin{example}
\(\PB_{A_{4}}\) has as generators
\[
\begin{array}{c}
\begin{tikzpicture}[scale=1.5]
	\node at (0,0) {$\scriptstyle 12$};
	\node at (0.5,0) {$\scriptstyle 23$};
	\node at (1,0) {$\scriptstyle 34$};
	\node at (1.5,0) {$\scriptstyle 45$};
	\node at (0.25,-0.25)  {$\scriptstyle 13$};
	\node at (0.75,-0.25)  {$\scriptstyle 24$};
	\node at (1.25,-0.25)  {$\scriptstyle 35$};
	\node at (0.5,-0.5) {$\scriptstyle 14$};
	\node at (1,-0.5) {$\scriptstyle 25$};
	\node at (0.75,-0.75)  {$\scriptstyle 15$};
\end{tikzpicture}		
\end{array}
=
\begin{array}{c}
\begin{tikzpicture}[scale=1.5]
	\node at (0,0) {$\scriptstyle a$};
	\node at (0.5,0) {$\scriptstyle b$};
	\node at (1,0) {$\scriptstyle c$};
	\node at (1.5,0) {$\scriptstyle d$};
	\node at (0.25,-0.25)  {$\scriptstyle e$};
	\node at (0.75,-0.25)  {$\scriptstyle f$};
	\node at (1.25,-0.25)  {$\scriptstyle g$};
	\node at (0.5,-0.5) {$\scriptstyle h$};
	\node at (1,-0.5) {$\scriptstyle i$};
	\node at (0.75,-0.75)  {$\scriptstyle j$};
\end{tikzpicture}		
\end{array} 
\]
As a slight abuse of notation in the sense of the introduction \(\ell_{ij}^{2}=[i,j]\). There are \(29\)  commutator relations, together with the 6 box relations
\begin{align*}
e^{-1}abcf^{-1}&=f^{-1}cbae^{-1}\\
f^{-1}bcdg^{-1}&=g^{-1}dcbf^{-1}\\
e^{-1}abgi^{-1}&=i^{-1}gbae^{-1}\\
h^{-1}afgi^{-1}&=i^{-1}gfah^{-1}\\
h^{-1}ecdg^{-1}&=g^{-1}dceh^{-1}\\
h^{-1}efdi^{-1}&=i^{-1}dfeh^{-1}.
\end{align*}
\end{example}
\subsection{Other Coxeter Types}
This section explains that  pure braid groups of other Coxeter arrangements are not in general generated by the analogue of \(\ell_{ij}^{2}.\)
Consider the Dynkin diagram
\[
\begin{tikzpicture}[scale=1.5] 
\node (A) at (-2,0) {$I_{n}$};
\node (B) at (-1.5,0) {\colonequals};
\node (C) at (-1,0) {$\circ$};
\node (D) at (0,0) {$\circ$};
\path[-,font=\scriptsize,>=angle 90] 
(C) edge node[above]{n} (D);				
\end{tikzpicture}
\]
with associated braid group \[
B_{I_{n}}\colonequals \left\langle s_{1}, s_{2} \left|
\begin{array}{l}
\underbrace{s_{1}s_{2}\hdots}_{n} =\underbrace{s_{2}s_{1}\hdots}_{n} 
\end{array}
\right.\right\rangle
\] and Weyl group \[W_{I_{n}}\colonequals \left\langle s_{1}, s_{2} \left|
\begin{array}{ll}
	s^{2}_{1}=s^{2}_{2}=1\\
\underbrace{s_{1}s_{2}\hdots}_{n} =\underbrace{s_{2}s_{1}\hdots}_{n}
\end{array}
\right.\right\rangle \cong D_{2n}, \]  where $D_{2n}$ is the dihedral group  of order $2n$. The pure braid group $\PB_{I_{n}}$  associated to the corresponding finite Coxeter group is still  the kernel of the natural map $ \upphi \colon B_{I_{n}} \twoheadrightarrow W_{I_{n}}, $ and is isomorphic to \(\pi_{1}(\mathbb{C}^{n}\setminus \bigcup_{i=1}^n (H_i)_{\mathbb{C}})\) where \(\scrH=\{H_i\}_{i=1}^n\) is the corresponding reflection arrangement.

Now, in general consider $\PB_{\mathcal{H}}\colonequals\pi_{1}(\mathbb{C}^{n}\setminus \bigcup H_{\mathbb{C}})$ for any Coxeter  hyperplane  arrangement $\mathcal{H}$, where $\bigcup H_{\mathbb{C}}$ is the union of complexified hyperplanes.  By \cite[Proposition 2.2(2)]{BMR} the abelianization of \(\PB_{\mathcal{H}}\) is the free abelian group over  a set of hyperplanes. 

\begin{remark}
In \cite{DW3}, it is proved that the derived category of a flopping contraction \(X \to \Spec R\) carries
an action of a subgroup of the fundamental group of the complexified complement of an associated \(\mathcal{H}\),	without  knowledge of the group presentation.  This subgroup \(K\) is defined to be generated by monodromy around all walls (including those of high codimension) from any fixed chamber. In type \(A\), this corresponds to the \(\ell_{i,j}^{2}\) from the earlier sections.
\end{remark}
 Write $H_{1}, \ldots, H_{n}$ for the \(n\) hyperplanes in \(\mathbb{R}^{2}\) associated to the Dynkin diagram $I_{n}$. Starting from a given chamber, we make a choice on the generators of the pure braid group $\PB_{I_{n}}$ by  finding the shortest  way to loop around each of  the hyperplanes; see e.g.\  \cite[Proposition 2.2(1)]{BMR}.

\begin{cor}\label{cor:not gen}
The pure braid group  $\PB_{I_{n}}$ has at least $n$ generators, and so  \(K \neq \PB_{I_{n}}\) whenever \(n \geq 4\).
\end{cor}
\begin{proof}
The number of generators of \(\PB_{I_{n}}\) is at least the number of generators for its abelianization. By \cite[Proposition 2.2(2)]{BMR} this is the number of hyperplanes, which is \(n.\) But $K$ has only $3$ generators since  the arrangement is in \(\mathbb{R}^{2}\)  so any chamber has only 3 walls: two of codimension one  and one of codimension two.
\end{proof}

	\end{document}